\documentclass{amsart}
\usepackage[all]{xy}
\usepackage{latexsym}
\usepackage{amsmath}
\usepackage{amssymb,amscd}
\usepackage{setspace}
\usepackage{mathdots}
\usepackage[mathscr]{eucal}
\usepackage{bbm}
\usepackage{stmaryrd}
\usepackage{color}
\usepackage[marginpar=2cm]{geometry}
\usepackage{amssymb,bm,amsmath, mathrsfs, amsxtra,amsthm,amscd }
\usepackage[all]{xy}

\usepackage{hyperref}
\usepackage{bbm}
\usepackage{pdfsync}

\usepackage[OT2,T1]{fontenc}

\newtheorem{theorem}{Theorem}[section] 
\newtheorem{lemma}[theorem]{Lemma}     
\newtheorem{corollary}[theorem]{Corollary}
\newtheorem{proposition}[theorem]{Proposition}
\newtheorem{remark}[theorem]{Remark}
\newtheorem{definition}[theorem]{Definition}
\newtheorem{conjecture}[theorem]{Conjecture}

\numberwithin{equation}{section}

\newcommand{\p}{\mathfrak{p}}

\newcommand{\cO}{\mathcal{O}}

\newcommand{\cS}{\mathcal{S}}
\newcommand{\cI}{\mathcal{I}}
\newcommand{\F}{\mathbb F}

\newcommand{\Q}{\mathbb Q}

\newcommand{\Z}{\mathbb Z}

\newcommand{\rsoc}{\mathrm{soc}}

\newcommand{\Fil}{\mathrm{Fil}}

\newcommand{\ra}{\rightarrow}

\newcommand{\GL}{\mathrm{GL}}
\newcommand{\SL}{\mathrm{SL}}

\newcommand{\bFp}{\overline{\F}_p}
\newcommand{\bQp}{\overline{\Q}_p}

\newcommand{\Sym}{\mathrm{Sym}}

\newcommand{\brho}{\overline{\rho}}

\newcommand{\xto}[1][]{\xrightarrow{#1}}
\newcommand{\simto}{
\xto[\sim]} 

\newcommand{\matr}[4]{\begin{pmatrix}{#1}&{#2}\\ {#3}&{#4}\end{pmatrix}}
\newcommand{\smatr}[4]{\bigl(\begin{smallmatrix} {#1}& {#2}\\ {#3}&{#4}\end{smallmatrix}\bigl)}

\def\QM{{\mathbb{Q}}}
\def\OC{{\mathcal{O}}}
\def\FM{{\mathbb{F}}}

\def\AM{{\mathbb{A}}}
\def\TM{{\mathbb{T}}}
\def\Fov{{\overline{F}}}
\def\into{\hookrightarrow}

 \def\CM{{\mathbb{C}}}
\def\RM{{\mathbb{R}}}
 \def\FC{{\mathcal{F}}}

\DeclareMathOperator{\Art}{{\mathrm{Art}}}

\DeclareMathOperator{\End}{{\mathrm{End}}}
\DeclareMathOperator{\Ext}{{\mathrm{Ext}}}
\DeclareMathOperator{\Fr}{{\mathrm{Fr}}}
\DeclareMathOperator{\Frob}{{\mathrm{Frob}}}
\DeclareMathOperator{\Gal}{{\mathrm{Gal}}}
\DeclareMathOperator{\Hom}{{\mathrm{Hom}}}

\DeclareMathOperator{\im}{{\mathrm{Im}}}

\DeclareMathOperator{\Ind}{{\mathrm{Ind}}}

\DeclareMathOperator{\JH}{{\mathrm{JH}}}

\DeclareMathOperator{\loc}{{\mathrm{loc}}}

\DeclareMathOperator{\Rep}{{\mathrm{Rep}}}

\DeclareMathOperator{\soc}{{\mathrm{soc}}}

\DeclareMathOperator{\tr}{{\mathrm{tr}}}

\def\G{\Gamma}

\def\s{\sigma}

\def\th{\theta}

\begin{document}

\title[]
{Multiplicity one for the mod $p$ cohomology of Shimura curves: the tame case }

\author{Yongquan HU \and Haoran WANG }
\date{}

\maketitle

\begin{abstract}
 Let $F$ be a totally real field, $\mathfrak{p}$ an unramified place of
$F$ dividing $p$ and $\overline{r}: \Gal(\overline{F}/F)\ra\GL_2(\bFp)$ a continuous irreducible modular representation. The work of Buzzard, Diamond and Jarvis \cite{BDJ} associates to $\overline{r}$ an admissible smooth representation of $\GL_2(F_\mathfrak{p})$  on the mod $p$ cohomology of Shimura curves attached to indefinite division algebras which split at $\mathfrak{p}$.
When $\overline{r}|_{\Gal(\overline{F_\mathfrak{p}}/F_\mathfrak{p})}$ is tamely ramified and generic (and under some additional technical assumptions), we determine the subspace of invariants of this representation under the principal congruence subgroup of level $\mathfrak{p}$. In particular, the subspace depends only on $\overline{r}|_{\Gal(\overline{F_\mathfrak{p}}/F_\mathfrak{p})}$ and satisfies a multiplicity one property.

\end{abstract}
\tableofcontents

\section{Introduction}

Let $p$ be a prime number. Barthel-Livn\'e \cite{BL} and Breuil \cite{Br03} gave a complete classification of irreducible   smooth representations of $\GL_2(\QM_p)$ over $\bFp$ with a central character,   which allowed Breuil to define a semi-simple mod $p$ local Langlands correspondence for $\GL_2(\QM_p)$ \cite{Br03}. The situation is much more complicated if one wants to establish a similar mod $p$ correspondence for $\GL_2(L)$ where $L\neq \QM_p$ is a finite extension of $\QM_p.$ There is no such complete classification, and the study of irreducible admissible smooth $\bFp$-representations of $\GL_2(L)$ becomes very subtle (\cite{BP}, \cite{Hu12}, \cite{Sch}).

Motivated by the local-global compatibility result of Emerton \cite{Em3} for the cohomology of modular curves, it seems to be very promising to seek for the hypothetical correspondence for $\GL_2(L)$ in the cohomology of Shimura curves.

More precisely, let $F$ be a totally real field, $F_v$ be the completion of $F$ at a fixed place $v$ of $F$ above $p$ with ring of integers $\cO_{F_v}$ and residue field $k_{F_v}.$ We consider the mod $p$ \'etale cohomology of a tower of Shimura curves $(X_U)_U$ over $F$ associated to an indefinite quaternion algebra $D$ with center $F$ which splits at all places over $p$ and at exactly one infinite place. Let $S^D(\bFp)$ denote the space
$$\varinjlim_U H^1_{\textrm{\'et}}(X_{U,\overline{F}},\bFp).$$
It is expected that one can get information on the hypothetical mod $p$ local Langlands correspondence by studying the action of $\GL_2(F_v)$ on $S^D(\bFp).$ For instance, assuming $F_v/\QM_p$ unramified, the $\GL_2(\OC_{F_v})$-socles of the irreducible $\GL_2(F_v)$ subrepresentations of  $S^D(\bFp)$ are described by the Buzzard-Diamond-Jarvis conjecture (\cite{BDJ}, proved in \cite{GLS}). Motivated by this, for a generic Galois representation $\brho:\Gal(\overline{F}_v/F_v)\to \GL_2(\bFp),$  Breuil and Pa\v{s}k\=unas \cite{BP} construct by local method an infinite family of smooth admissible $\bFp$-representations of $\GL_2(F_v)$ whose $\GL_2(\OC_{F_v})$-socles are as predicted by \cite{BDJ}. Furthermore, they conjecture that if $\overline{r}$ is a globlisation of $\brho$ to a modular Galois representation of $\Gal(\overline{F}/F),$ then the $\overline{r}$-isotypic part of $S^D(\bFp)$ contains one of these representations. This conjecture is recently proved by Emerton, Gee and Savitt in \cite{EGS} under mild Taylor-Wiles type hypothesis.

In this paper, we give further constraint on these $\GL_2(F_v)$-representations in the case where $\brho$ is a tamely ramified (that is either split or irreducible) and generic representation. To state our main theorem, we introduce some notations. Let $\overline{r}:\Gal(\overline{F}/F)\to \GL_2(\bFp)$ be a globlisation of $\brho,$ which is continuous irreducible and totally odd. We assume that $\overline{r}$ is modular in the sense that
\[
\pi^D(\overline{r}):=\Hom_{\bFp[\Gal(\overline{F}/F)]}(\overline{r},S^D(\bFp))\neq 0.
\]
We can use the action of Hecke operators away from $v$ on $\pi^D(\overline{r})$ to define a local factor $\pi_v^D(\overline{r})$ at $v,$ which is an admissible smooth representation of $\GL_2(F_v),$ and is supposed to be the right candidate in the mod $p$ local Langlands correspondence, and many important properties about it have been established, see e.g. \cite{Gee11}, \cite{BD11}, \cite{EGS}. Our main result is the following theorem.

\begin{theorem}\label{theorem-main}
Assume that $F_v$ is unramified over $\QM_p$ and $\brho$ is tamely ramified and generic. Under certain assumptions (see Cor. \ref{cor-final}), we have $\pi_v^D(\overline{r})^{K_1}\cong D_0(\brho)$ as $\GL_2(\OC_{F_v})$-representations, where $K_1=\ker(\GL_2(\OC_{F_v})\to \GL_2(k_{F_v})),$ and $D_0(\brho)$ is the finite dimensional $\GL_2(\OC_{F_v})$-representation attached to $\brho$ in \cite[Thm. 1.1]{BP}.
\end{theorem}

\begin{remark}
After the first version of the paper was written, we are informed that Le, Morra and Schraen obtain a similar result independently \cite{LMS}. Both the proofs use results of \cite{EGS} as a global input, however, the local representation theory part is quite different.
\end{remark}

We also prove a similar result when $D$ is a definite quaternion algebra unramified at all places over $p.$
In \cite{EGS}, the subspace of  pro-$p$-Iwahori fixed vectors   of $\pi_v^D(\overline{r})$ is determined, for $\brho$ generic, but could be reducible non-split.  The proof of our theorem uses the construction of \cite{EGS} as a main tool. We intend to extend the result to reducible non-split $\brho$ in future work.

The organization of the paper is as follows.  In Section 2, we give all the local results that we need, especially the local criterion Corollary \ref{cor-local result}. Note that the criterion does not apply to $\brho$ reducible non-split (see Remark \ref{rem-nonsplit}). In Section 3, we deduce our main theorem from our local results and results of \cite{EGS}.

We assume $p\geq 3$ throughout the paper; this is harmless for our application. We fix a finite extension $E$ of $\QM_p$ with ring of integers $\OC_E,$ uniformiser $\varpi_E$ and residue field $\FM,$ which is allowed to be enlarged. They will serve as coefficient fields (or rings) for our representations.

\section{Local input}

In this section, $L$ denotes a finite extension of $\Q_p$, with ring of integers $\cO_L$, maximal ideal $\p_L$, and residue field (identified with) $\F_q=\F_{p^f}$.   Fix a uniformiser $\varpi_L$ of $L$; we take $\varpi_L=p$ when $L$ is unramified over $\Q_p$. For $\lambda\in\F_q$, $[\lambda]\in\cO_L$ denotes its Teichm\"uller lift.
Let $\Gamma=\GL_2(\F_q)$. 
We call a weight an   irreducible representation of $\Gamma$ over $\bFp$. We take $E$ large enough so that any weight is defined over $\F$, then a weight is (up to isomorphism) of the form (\cite[Prop. 1]{BL})
\[\Sym^{r_0}\F^2\otimes_{\F}(\Sym^{r_1}\F^2)^{\mathrm{Fr}}\otimes_{\F}\cdots\otimes_{\F}(\Sym^{r_{f-1}}\F^2)^{\mathrm{Fr}^{f-1}}\otimes_{\F}{\det}^a\]
where  $0\leq r_i\leq p-1$, $0\leq a\leq q-2$ and $\mathrm{Fr}: \smatr{a}bcd\mapsto \smatr{a^p}{b^p}{c^p}{d^p}$ is the Frobenius on $\Gamma$.  We denote this representation by $(r_0,\cdots,r_{f-1})\otimes {\det}^a$.  To simplify the notation, we write $V_r:=\Sym^r\F^2$ for $0\leq r\leq p-1$.
We recall the following definition  from \cite[Def. 2.1.4]{EGS}.

\begin{definition}\label{defn-regular}
We say that $(r_0,\cdots,r_{f-1})\otimes {\det}^a$ is \emph{regular} if no $r_i$ is equal to $p-1$.
\end{definition}

Let $\Rep_{\Gamma}$ be the category of finite dimensional $\F$-representations of $\Gamma$. If $M\in \Rep_{\Gamma}$,  we write $\{\Fil_iM, i\geq 0\}$ for its socle filtration, that is, $\Fil_0M=\rsoc_{\Gamma}M$, $\Fil_1M$ is the preimage in $M$ of $\rsoc_{\Gamma}(M/\Fil_0M)$, etc. We call $\mathrm{gr}_iM:=\Fil_iM/\Fil_{i-1}M$ the $i$-th layer of the socle filtration where $\Fil_{-1}M:=0$ by convention.   Similarly we write $\{\Fil^iM, i\geq 0\}$ for the cosocle filtration of $M$ and $\mathrm{gr}^iM$ for the graded pieces. 

We first recall some results on the structure of injective envelopes in $\Rep_{\Gamma}$, mainly following  \cite[\S3]{BP}.
If $r=p-1$, we set $R_{p-1}:=V_{p-1}$.
If $0\leq r\leq p-2$, let $R_r$ be a $\Gamma$-invariant  subspace of $V_{p-r-1}\otimes V_{p-1}$ defined in \cite[Def. 4.2.10]{Pa04}. By \cite[Lem. 3.1]{AJL}, it is self-dual up to a twist and admits $V_r\otimes{\det}^{p-1-r}$ as a sub-representation and also as a quotient. Moreover, letting $W_r\subset R_r$ be given by the exact sequence
\begin{equation}\label{equation-p_i}0\ra W_r\ra R_{r}  \ra  V_r\otimes{\det}^{p-1-r}\ra0,\end{equation}
 we get a filtration on $R_r$:
 \[0\subsetneq V_r\otimes{\det}^{p-1-r}\subsetneq W_r\subsetneq R_r\]
with graded pieces being $V_r\otimes{\det}^{p-1-r}$, $V_{p-2-r}\otimes V_1^{\rm Fr}$, and $V_r\otimes {\det}^{p-1-r}$. Remark that when $f\geq 2$ or $f=1$ and $r\neq0$, this coincides with the socle filtration of $R_{\sigma}$.

Now  fix a weight $\sigma=(r_0,\cdots,r_{f-1})\otimes {\det}^{-\sum_{i=0}^{f-1}p^ir_i}$  with $0\leq r_i\leq p-1$ and set $R_{\sigma}:=\otimes_{i=0}^{f-1}R_{r_i}^{\mathrm{Fr}^i}$.
If $\dim \sigma\geq 2$, i.e., if not all $r_i$ equal to $0$, then $R_{\sigma}$ is an injective envelope of $\sigma$, see \cite[Cor. 4.2.22]{Pa04}. Otherwise, $R_{
\sigma}$ is isomorphic to ${\rm inj}_{\Gamma}(0,\cdots,0)\oplus (p-1,\cdots,p-1)$, where ${\rm inj}_{\Gamma}(0,\cdots,0)$ denotes the injective envelope of $(0,\cdots,0)$ in $\Rep_{\Gamma}$, see \cite[Cor. 4.2.31]{Pa04}.

From now on (until the end of subsection \ref{subsection-A}), we assume  that  $\dim \sigma\geq 2$ and that $\sigma$ is \emph{regular} in the sense of Definition \ref{defn-regular}.


\begin{lemma}\label{lemma-inter}
(i) For $0\leq i\neq j\leq f-1$, we have
\begin{equation}\label{equation-inter}(W_{r_i}^{\Fr^i}\otimes R_{r_j}^{\Fr^j})\cap (R_{r_i}^{\Fr^i}\otimes W_{r_j}^{\Fr^j})=W_{r_i}^{\Fr^i}\otimes W_{r_j}^{\Fr^j}.\end{equation}
The same statement holds if we replace $W_{r_i}$, $W_{r_j}$ by $V_{r_i}\otimes{\det}^{p-1-r_i}$, $V_{r_j}\otimes{\det}^{p-1-r_j}$.

(ii) We have $\bigcap_{i=0}^{f-1} \bigl(W_{r_i}^{\Fr^i}\otimes (\otimes_{j\neq i}R_{r_j}^{\Fr^j})\bigr)=\otimes_{i=0}^{f-1} W_{r_i}^{\Fr^i}$. The same statement holds if we replace $W_{r_i}$ by $V_{r_i}\otimes{\det}^{p-1-r_i}$.
\end{lemma}
\begin{proof}
The (i) is standard, see for example \cite[Chap. I, \S2, n$^{\circ}$ 6, Prop. 7]{Bour}.
For (ii), we  proceed by an obvious induction  using (i).
\end{proof}

We find it convenient to give a name for the intersection in Lemma \ref{lemma-inter}(ii).
\begin{definition}
Denote by $A_{\sigma}$ the following sub-representation of $R_{\sigma}$:
\[A_{\sigma}=\bigotimes_{i=0}^{f-1}W_{r_i}^{\Fr^i}.\]
\end{definition}
\begin{proposition}\label{prop-A-multone}
$A_{\sigma}$ is multiplicity free, and is the largest sub-representation of $R_{\sigma}$ which is multiplicity free.
\end{proposition}
\begin{proof}
In the notation of \cite[\S3]{BP}, $A_{\sigma}$ is exactly the representation $V_{\mathbf{2p-2-r}}$ defined in \cite[Def. 3.3]{BP}. The assertion then follows from \cite[Prop. 3.6,  Cor. 3.11]{BP}.
\end{proof}

  To study the structure of $A_{\sigma}$, we need to introduce another sub-representation of $R_{\sigma}$.

\subsection{The representation $A'_{\sigma}$}

 \begin{definition}
We define $A'_{\sigma}$ to be the sub-representation of $R_{\sigma}$:
\[A'_{\sigma}:=\sum_{i=0}^{f-1}\bigl(R_{r_i}^{\Fr^i}\otimes (\otimes_{j\neq i}(V_{r_j}\otimes{\det}^{p-1-r_j})^{\Fr^j})\bigr).\]
We write $A_{\sigma,i}'$ for $R_{r_i}^{\Fr^i}\otimes (\otimes_{j\neq i}(V_{r_j}\otimes{\det}^{p-1-r_j})^{\Fr^j})$ so that $A'_{\sigma}=\sum_{i=0}^{f-1}A'_{\sigma,i}$.
\end{definition}

\begin{lemma}\label{lemma-f+1}
The multiplicity with which $\sigma$ appears in (the semisimplification of) $A'_{\sigma}$ is $f+1$.
\end{lemma}
\begin{proof}
If $f=1$,  then $A'_{\sigma,0}$ is just $R_{r_0}=R_{\sigma}$ and the result follows from \cite[Lemmas 3.4, 3.5, 3.8(i)]{BP}. Assume $f\geq 2$ in the rest of the proof. Tensoring with $\otimes_{k\neq i,j}(V_{r_k}\otimes{\det}^{p-1-r_k})^{\Fr^k}$, Lemma \ref{lemma-inter}(i) implies that for any $0\leq i\neq j\leq f-1$, we have
\[A'_{\sigma,i}\cap A'_{\sigma,j}=\sigma,\] so that
\[0\ra \sigma\ra A'_{\sigma,i}\oplus A'_{\sigma,j}\ra A'_{\sigma,i}+A'_{\sigma,j}\ra 0 \]
is exact. This shows that the multiplicity with which $\sigma$ appears in $A'_{\sigma,i}+A'_{\sigma,j}$ is $3$.  An induction shows that the multiplicity with which $\sigma$ appears in $\sum_{i=0}^{m}A'_{\sigma,i}$ is $m+1$ for any $0\leq m\leq f-1$.
\end{proof}

We can describe $A'_{\sigma,i}$ explicitly.
\begin{lemma}\label{lemma-cosocle-A_i}
The $\Gamma$-socle and cosocle of $A'_{\sigma,i}$ is isomorphic to $\sigma$.
\end{lemma}
\begin{proof}

Since $A_{\sigma,i}'$ is a non-zero sub-representation of  $R_{\sigma}$ and $\rsoc(R_{\sigma})=\sigma$, we have $\rsoc(A_{\sigma,i}')=\sigma$ and in particular is irreducible.
Now, the representations $R_{r_i}^{\Fr^i}$ and $(V_{r_j}\otimes{\det}^{p-1-r_j})^{\Fr^j}$ are  self-dual up to a twist, so is $A'_{\sigma,i}$. We deduce that the cosocle of $A'_{\sigma,i}$ is also irreducible.  It follows from (\ref{equation-p_i}) that $A'_{\sigma,i}$ admits $\sigma$ as a quotient, hence the cosocle has to be $\sigma$.
\end{proof}

\begin{definition}
For each $0\leq i\leq f-1$, we define two weights as follows:
\begin{enumerate}
\item[--] if $f=1$, let $\mu_0^{+}(\sigma):=V_{p-1-r_0}$ and $\mu_0^-(\sigma):=V_{p-3-r_0}\otimes{\det}$;
\item[--] if $f\geq 2$, let \[\mu_i^{\pm}(\sigma):=(r_0,\cdots,p-2-r_{i},r_{i+1}\pm1,\cdots,r_{f-1})\otimes{\det}^{a_i^{\pm}-\sum_{j=0}^{f-1}p^jr_j}.\]
where $a_i^+=p^i(r_i+1)-p^{i+1}$ and $a_i^-=p^i(r_i+1)$.
\end{enumerate}
By convention, when $\mu_i^{-}(\sigma)$ is not a genuine weight, i.e. when $f=1$ and $r_0=p-2$, or $f\geq 2$ and $r_{i+1}=0$,  we say that $\mu_i^-(\sigma)$ is not defined.
\end{definition}

By \cite[Lem. 3.8(i)]{BP} the representation \[(V_{p-2-r_i}^{\mathrm{Fr}^i}\otimes V_1^{\mathrm{Fr}^{i+1}})\otimes\bigl(\otimes_{j\neq i}(V_{r_j}\otimes{\det}^{p-1-r_j})^{\Fr^j}\bigr) \]
is semi-simple, and is isomorphic to $\mu_i^+(\sigma)\oplus\mu_i^-(\sigma)$; here and below we ignore $\mu_i^-(\sigma)$ in the direct sum when it is not defined. Knowing the socle and cosocle of $A'_{\sigma,i}$, we get the following.

\begin{lemma}
The socle filtration of $A'_{\sigma,i}$ has length 3, with graded pieces  $\sigma$, $\mu_i^+(\sigma)\oplus\mu_i^-(\sigma)$, and $\sigma$.
\end{lemma}

\begin{corollary}\label{cor-A'_i-quotient}
We have an isomorphism $A_{\sigma,i}'/(A_{\sigma,i}'\cap A_{\sigma})\cong \sigma$.
\end{corollary}
\begin{proof}
It follows from that $A'_{\sigma,i}\cap A_{\sigma}$ is of multiplicity free (by Proposition \ref{prop-A-multone}) and contains  $\Fil_1A_{\sigma,i}'$ by construction.
\end{proof}

 \subsection{The structure of $
A_{\sigma}$} \label{subsection-A}

\begin{proposition}\label{prop-socleB}
We have $\rsoc_{\Gamma}(R_{\sigma}/A_{\sigma})\cong \sigma^{\oplus f}$.
\end{proposition}

\begin{proof}
For each $0\leq i\leq f-1$, tensoring the exact sequence (\ref{equation-p_i}) with $\otimes_{j\neq i}R_{r_j}^{\Fr^j}$ gives an exact sequence
\[0\ra W_{r_i}^{\Fr^i}\otimes (\otimes_{j\neq i}R_{r_j}^{\Fr^j})\ra R_{\sigma}\ra (V_{r_i}\otimes{\det}^{p-1-r_i})^{\Fr^i}\otimes (\otimes_{j\neq i}R_{r_j}^{\Fr^j})\ra 0.\]
By Lemma \ref{lemma-inter}(ii), $A_{\sigma}$ is equal to the intersection of  $W_{r_i}^{\Fr^i}\otimes (\otimes_{j\neq i}R_{r_j}^{\Fr^j})$ for  $0\leq i\leq f-1$, so we get an injection
 \[R_{\sigma}/A_{\sigma}\hookrightarrow \bigoplus_{i=0}^{f-1}(V_{r_i}\otimes{\det}^{p-1-r_i})^{\Fr^i}\otimes (\otimes_{j\neq i}R_{r_j}^{\Fr^j}).\]
Since $(V_{r_i}\otimes{\det}^{p-1-r_i})^{\Fr^i}\otimes(\otimes_{j\neq i}R_{r_j}^{\Fr^j})$  embeds into $R_{\sigma}$, we deduce that $R_{\sigma}/A_{\sigma}$ embeds into $(R_{\sigma})^{\oplus f}$ so that $\rsoc_{\Gamma}(R_{\sigma}/A_{\sigma})\subseteq \sigma^{\oplus f}$. 
To show the equality, we use the embedding
\[A'_{\sigma}/(A'_{\sigma}\cap A_{\sigma})\hookrightarrow R_{\sigma}/A_{\sigma},\]
whose image in fact lies in the \emph{socle} of $R_{\sigma}/A_{\sigma}$ by Corollary \ref{cor-A'_i-quotient}. Since $A_{\sigma}$ is multiplicity free, $\sigma$ appears in $A'_{\sigma}\cap A_{\sigma}$ exactly once. Lemma \ref{lemma-f+1} implies that the multiplicity of $\sigma$ in   $\rsoc_{\Gamma}(R_{\sigma}/A_{\sigma})$ is at least $f$, which completes the proof.
\end{proof}

It is easy to see from the above proof that $A'_{\sigma}/(A'_{\sigma}\cap A_{\sigma})$  is identified with $\rsoc_{\Gamma}(R_{\sigma}/A_{\sigma})$. Let $B_{\sigma}$ be the pullback of $\rsoc_{\Gamma}(R_{\sigma}/A_{\sigma})$, i.e. given by
\[0\ra A_{\sigma}\ra B_{\sigma}\ra \rsoc_{\Gamma}(R_{\sigma}/A_{\sigma})\ra0.\]
Then $A'_{\sigma}$ is contained in $B_{\sigma}$; in fact, $B_{\sigma}=A_{\sigma}'+A_{\sigma}$.
\begin{lemma}\label{lemma-nosigma}
The quotient $B_{\sigma}/A'_{\sigma}$ has no sub-quotient isomorphic to $\sigma$.
\end{lemma}
\begin{proof}
We have an exact sequence $0\ra A_{\sigma}\ra B_{\sigma}\ra \sigma^{\oplus f}\ra0$ by Proposition \ref{prop-socleB}. Moreover,   $A_{\sigma}$ is multiplicity free, with $\sigma$ appearing there exactly once. So the multiplicity of $\sigma$ in (the semisimplification of) $B_{\sigma}$ is $f+1$, which is also the multiplicity of $\sigma$ in $A_{\sigma}'$ by Lemma \ref{lemma-f+1}.
\end{proof}
\begin{proposition}\label{prop-Q'}
Let $M$ be  a sub-representation of $B_{\sigma}$. Assume $M\cap A'_{\sigma}\subseteq A_{\sigma}$. Then $M\subseteq A_{\sigma}$.
\end{proposition}
\begin{proof}
Consider the commutative diagram of exact sequences (with $\alpha$ the induced morphism)
\[\xymatrix{0\ar[r]& M\cap A_{\sigma}\ar[d]\ar[r]& M\ar[r]\ar[d]& M/(M\cap A_{\sigma})\ar^{\alpha}[d]\ar[r]&0\\
0\ar[r]&A_{\sigma}\ar[r]&B_{\sigma}\ar[r]&\rsoc_{\Gamma}(R_{\sigma}/A_{\sigma})\ar[r]&0.}\]
Since $\alpha$ is injective, it suffices to show $\alpha=0$.
The assumption implies $M\cap A'_{\sigma}\subseteq M\cap A_{\sigma}$. Denote by $\alpha'$ the composition
\[\alpha': M/(M\cap A'_{\sigma})\twoheadrightarrow M/(M\cap A_{\sigma})\overset{\alpha}{\ra} \rsoc_{\Gamma}(R_{\sigma}/A_{\sigma}).\]
The first map being surjective, $\alpha=0$ if and only if $\alpha'=0$. However, the natural isomorphism $M/(M\cap A'_{\sigma})\cong (M+A'_{\sigma})/A'_{\sigma}$ allows us to view $M/(M\cap A'_{\sigma})$ as a sub-representation of $B_{\sigma}/A'_{\sigma}$. By Lemma \ref{lemma-nosigma}, $B_{\sigma}/A'_{\sigma}$ has no sub-quotient isomorphic to $\sigma$, while $\rsoc_{\Gamma}(R_{\sigma}/A_{\sigma})\cong \sigma^{\oplus f}$ by Proposition \ref{prop-socleB}, hence $\alpha'=0$.
\end{proof}

Now we treat the general case.
\begin{proposition}\label{prop-criteria}
Let $M$ be a sub-representation of $R_{\sigma}$. Assume $M\cap A'_{\sigma}\subseteq A_{\sigma}$. Then $M\subseteq A_{\sigma}$.
\end{proposition}
\begin{proof}
Let $M':=M\cap B_{\sigma}$.  Then $M'\cap A'_{\sigma}=M\cap A'_{\sigma}\subseteq A_{\sigma}$, so Proposition \ref{prop-Q'} implies $M'\subseteq A_{\sigma}$. Consider the following commutative diagram where $\alpha'$ is the map constructed in  Proposition \ref{prop-Q'} (with $M$ replaced by $M'$)
\[\xymatrix{M'\ar@{^{(}->}[r]\ar@{->>}[d]&M\ar@{->>}[d]\\
M'/(M'\cap A'_{\sigma}) \ar^{\alpha'}[d] &M/(M\cap A_{\sigma})\ar@{^{(}->}^{\alpha}[d]\\
\rsoc_{\Gamma}(R_{\sigma}/A_{\sigma})\ar@{^{(}->}[r]&R_{\sigma}/A_{\sigma}.}\]
We know that $\alpha'=0$ and we want to prove $\alpha=0$. Assume $\alpha\neq0$. Then $\im(\alpha)$ would have a non-zero intersection with $\rsoc_{\Gamma}(R_{\sigma}/A_{\sigma})$. Let $\bar{v}\in \im(\alpha)\cap \rsoc_{\Gamma}(R_{\sigma}/A_{\sigma})$ and let $v\in M$ be a lift of $\bar{v}$. Then $v$ belongs to $B_{\sigma}$ by definition, so that $v\in M\cap B_{\sigma}=M'$, which gives a contradiction because $\alpha'=0$.
\end{proof}

Next we give a criterion for the condition in Proposition \ref{prop-criteria} to hold.
In particular, if $M$ is a sub-representation of $R_{\sigma}$, then $\mathrm{gr}_1M$ embeds into
\[\mathrm{gr}_1R_{\sigma}\cong\bigoplus_{i=0}^{f-1} \bigl(\mu_i^+(\sigma)\oplus \mu^-_{i}(\sigma)\bigr).\]
\begin{definition}\label{defn-alternative}
Let $\sigma$ be a regular weight and $M$ be a sub-representation of $R_{\sigma}$.
We say that $M$ is \emph{alternative},   if
  for each $0\leq i\leq f-1$, $\mathrm{gr}_1M$ does not contain simultaneously $\mu_i^+(\sigma)$ and $\mu_i^-(\sigma)$ (resp. does not contain $\mu_i^+(\sigma)$) if $\mu_i^-(\sigma)$ is defined (resp. if $\mu_i^-(\sigma)$ is not defined).
\end{definition}

Since  $R_{\sigma}$ is an injective envelope of $\sigma$ in $\Rep_{\Gamma}$, we have  a natural isomorphism $\Hom_{\Gamma}(\sigma,R_{\sigma}/A_{\sigma})\cong\Ext^1_{\Gamma}(\sigma,A_{\sigma})$. Hence  $\Ext^1_{\Gamma}(\sigma,A_{\sigma})$  is of dimension $f$  over $\F$ by Proposition \ref{prop-socleB}. On the other hand, for each $0\leq i\leq f-1$, we can use Corollary \ref{cor-A'_i-quotient} to view  $A'_{\sigma,i}$ as an (non-zero) element in $\Ext^1_{\Gamma}(\sigma,A_{\sigma}\cap A'_{\sigma,i})$, whose push-out provides a non-zero element  in $\Ext^1_{\Gamma}(\sigma,A_{\sigma})$, denoted by $X_i$, and it is easily seen that $\{X_i, 0\leq i\leq f-1\}$ form an $\F$-basis of $\Ext^1_{\Gamma}(\sigma,A_{\sigma})$. For our purpose, it is more convenient to view $X_i$ as  an element of $\Ext^1_{\Gamma}(\sigma,A_{\sigma}\cap A'_{\sigma})$. Remark, to be precise, that $A_{\sigma}\cap A'_{\sigma}$ fits in a short exact sequence
\begin{equation}\label{equation-A-A'}0\ra \sigma\ra A_{\sigma}\cap A'_{\sigma}\ra \bigoplus_{i=0}^{f-1}(\mu_i^+(\sigma)\oplus\mu_i^-(\sigma))\ra0,\end{equation}
and by Lemma \ref{lemma-cosocle-A_i} we have $\Hom_{\Gamma}(X_i,\mu_j^{\pm}(\sigma))=0$ if and only if $i=j$.

\begin{lemma}\label{lemma-alternative}
Let $M$ be an alternative  sub-representation of $R_{\sigma}$. 
Then $M\cap A'_{\sigma}\subseteq A_{\sigma}$.
\end{lemma}

\begin{proof}
 Obviously we may assume that $M$ is contained in $A_{\sigma}'$. Suppose that $M\nsubseteq A_{\sigma}$, so that $M/(A_{\sigma}\cap M)$ is non-zero. Since every irreducible sub-representation of $M/(A_{\sigma}\cap M)$ is isomorphic to $\sigma$,  we get a non-split extension
\[0\ra  A_{\sigma}\cap M\ra Y\ra \sigma\ra0\]
whose push-out gives an extension class in $\Ext^1_{\Gamma}(\sigma,A_{\sigma}\cap A'_{\sigma})$,  denoted   by $X$. Then there exist $c_i\in\F$, not all zero, such that
\[X=c_0X_0+\cdots+c_{f-1}X_{f-1}.\]
Without loss of generality, we assume that $c_0\neq 0$. Since $\Hom_{\Gamma}(X_0,\mu_0^{\pm}(\sigma))=0$ by Lemma \ref{lemma-cosocle-A_i}, and
\[\Hom_{\Gamma}(X_i,\mu_0^{\pm}(\sigma))\neq0,\ \ 1\leq i\leq f-1,\] Lemma \ref{lemma-Baer} below, applied to $A=\sigma$, $B=A_{\sigma}\cap A_{\sigma}'$, $C=\mu_0^+(\sigma)$ or $\mu_0^-(\sigma)$,  implies that
\begin{equation}\label{equation-X}\Hom_{\Gamma}(X, \mu_0^{\pm}(\sigma))=0.\end{equation}
On the other hand, since $M$ is alternative,  $\mathrm{gr}_1M$ does not contain $\mu_0^*(\sigma)$ for $*\in \{+,-\}$ (with $*=+$ if $\mu_0^-(\sigma)$ is not defined), i.e. $\Hom_{\Gamma}(A_{\sigma}\cap M,\mu_0^*(\sigma))=0$. Using \eqref{equation-A-A'} we see that
\[\Hom_{\Gamma}\bigl((A_{\sigma}\cap A'_{\sigma})/(A_{\sigma}\cap M),\mu_0^*(\sigma)\bigr)\neq0.\]
By construction we have a commutative diagram of exact sequences
\[\xymatrix{0\ar[r]& A_{\sigma}\cap M\ar[r]\ar[d]&Y\ar[r]\ar[d]&\sigma\ar[r]\ar@{=}[d]&0\\
0\ar[r]&A_{\sigma}\cap A_{\sigma}'\ar[r]&X\ar[r]&\sigma\ar[r]&0}\]
from which we get $(A_{\sigma}\cap A_{\sigma}')/(A_{\sigma}\cap M)\cong X/Y$, therefore
\[0\neq \Hom_{\Gamma}(X/Y,\mu_0^*(\sigma))\hookrightarrow  \Hom_{\Gamma}(X,\mu_0^*(\sigma)).\] This contradicts \eqref{equation-X} and allows to conclude.
\end{proof}

\begin{lemma}\label{lemma-Baer}
Let $A,B,C\in\Rep_{\Gamma}$. Let $X,X'\in\Ext^1_{\Gamma}(A,B)$ which are represented by short exact sequences $0\ra B\ra X\ra A\ra0$ and $0\ra B\ra X'\ra A\ra0$. Assume that

(i) $\Hom_{\Gamma}(X,C)=0$;

(ii)  $\Hom_{\Gamma}(X',C)\ra \Hom_{\Gamma}(B,C)$ is an isomorphism.

\noindent Then the Baer sum $X+X'\in\Ext^1_{\Gamma}(A,B)$ satisfies $\Hom_\G(X+X',C)=0$.
\end{lemma}
\begin{proof}
We first recall the construction of the Baer sum. Let $X''$ be the pullback $\{(x,x')\in X\times X': \bar{x}=\bar{x}'\ \mathrm{in}\ A\}$. Then $X+X'$ is the quotient of $X''$  by the skew diagonal $\{(-b,b):b\in B\}$, see \cite[Def. 3.4.4]{We}.
$X''$ contains three copies of $B$:  $0\times B$, $B\times 0$, and the skew diagonal,  taking quotient by  which we get respectively $X$, $X'$ and $Y:=X+X'$. Denote by $s:B\ra X''$ the skew diagonal morphism.

On the one hand, we have a commutative diagram of exact sequences
\[\xymatrix{0\ar[r]& B\ar[r]&X'\ar[r] &A\ar[r]&0\\
0\ar[r]&B\ar@{=}[u]\ar^{s}[r]&X''\ar[u]\ar[r]&Y\ar[u]\ar[r]&0}\]
which induces (using (ii))
\[\xymatrix{\Hom_{\Gamma}(X',C)\ar@{=}[r]\ar[d]&\Hom_{\Gamma}(B,C)\ar@{=}[d]\\
\Hom_{\Gamma}(X'',C)\ar^{\tilde{s}}[r]&\Hom_{\Gamma}(B,C).}\]
In particular, $\tilde{s}$ is surjective.
On the other hand,  using (i) and the exact sequence
$0\ra 0\times B\ra X''\ra X\ra0$
we get   \[\dim_{\F} \Hom_{\Gamma}(X'',C)\leq \dim_{\F}\Hom_{\Gamma}(B,C),\]
hence $\tilde{s}$ is in fact an isomorphism.

 The lemma follows by applying $\Hom_{\Gamma}(*,C)$ to the sequence $0\ra B\overset{s}{\ra} X''\ra Y\ra0$.
\end{proof}

\begin{remark}
Although we have only defined $R_{\sigma}$, $A_{\sigma}$, etc and proved the results for regular weights $\sigma$ of the form $(r_0,\cdots,r_{f-1})\otimes{\det}^{-\sum_{i=0}^{f-1}p^ir_{i}}$ such that $\dim\sigma\geq 2$, all these can obviously be generalized to   general regular weights of dimension at least $2$ by a twist.
\end{remark}

\subsection{An auxiliary lemma}

Let $\sigma$ be a regular weight (possibly of dimension $1$). By \cite[Lem. 3.2]{BP}, the irreducible constituents of $\mathrm{inj}_{\Gamma}{\sigma}$ (without multiplicities) are parametrised by a certain set $\cI(x_0,\cdots,x_{f-1})$ defined in the beginning of \cite[\S3]{BP}. For later use, we recall its definition. The set $\cI(x_0,\cdots,x_{f-1})$ consists  of elements of the form $\lambda=(\lambda_0(x_0),\cdots,\lambda_{f-1}(x_{f-1}))$ where $\lambda_0(x_0)\in\{x_0,p-2-x_0\pm1\}$ if $f=1$, and if $f>1$ then: \begin{enumerate}
\item[(i)] $\lambda_i(x_i)\in\{x_i,x_i\pm1,p-2-x_i,p-2-x_i\pm1\}$ for $0\leq i\leq f-1$
\item[(ii)] if $\lambda_i(x_i)\in\{x_i,x_i\pm1\}$, then $\lambda_{i+1}(x_{i+1})\in\{x_{i+1},p-2-x_{i+1}\}$
\item[(iii)] if $\lambda_i(x_i)\in\{p-2-x_i,p-2-x_i\pm1\}$, then $\lambda_{i+1}(x_{i+1})\in\{x_{i+1}\pm1,p-2-x_{i+1}\pm1\}$
\end{enumerate}
with the conventions $x_{f}=x_0$ and $\lambda_{f}(x_f)=\lambda_0(x_0)$.
Each $\lambda$ gives rise to a weight by ``evaluating'' at $\sigma=(r_0,\cdots,r_{f-1})\otimes\eta$, but note that it is not a genuine one if $\lambda_i(r_i)<0$ or $\lambda_i(r_i)>p-1$ for some $i$ (so we will ignore it in this case). In any case, an irreducible constituent of $A_{\sigma}$ corresponds to a unique element of $\cI(x_0,\cdots,x_{f-1})$.

Let $\cS:=\{0,...,f-1\}$. For $\lambda\in\cI(x_0,\cdots,x_{f-1})$, set
\[\cS(\lambda):=\{i\in\cS\ |\ \lambda_i(x_i)=p-2-x_i\pm1, x_i\pm1\}.\]
We also write $\cS(\tau):=\cS(\lambda)$ if $\tau$ is the corresponding weight.
Recall from \cite[Def. 4.10]{BP} that, if $\lambda,\lambda'\in\cI(x_0,\cdots,x_{f-1})$, we say $\lambda$ and $\lambda'$ are \emph{compatible} if, whenever $\lambda_i(x_i)\in \{p-2-x_i-\pm1,x_i\pm1\}$ and $\lambda_i'(x_i)\in\{p-2-x_i-\pm1,x_i\pm1\}$ for the same $i$, then the signs of the $\pm1$ are the same in $\lambda_i(x_i)$ and $\lambda'_i(x_i)$.
The following property is directly checked.
\begin{lemma}\label{lemma-oneweight}
Fix an element $\lambda\in\cI(x_0,\cdots,x_{f-1})$. For every subset $\cS'$ of $\cS(\lambda)$, there exists exactly one element $\lambda'\in\cI(x_0,\cdots,x_{f-1})$ such that $\cS(\lambda')=\cS'$ and $\lambda'$ is compatible with $\lambda$.
 \end{lemma}

Let $\tau$ be an irreducible constituent  of $\mathrm{inj}_{\Gamma}\sigma$. Then there exist finite dimensional representations of $\Gamma$ with socle $\sigma$ and cosocle $\tau$ by taking the image of any non-zero morphism $\phi\in\Hom_{\Gamma}(\mathrm{inj}_{\Gamma}\tau,\mathrm{inj}_{\Gamma}\sigma)$. Since $\sigma$ is regular, \cite[Cor. 3.12]{BP}  implies that among these representations there exists a unique one,  denoted by $I(\sigma,\tau)$, such that $\sigma$ appears with multiplicity 1. Moreover,  $I(\sigma,\tau)$ is multiplicity free.
By \cite[Cor. 4.11]{BP} and Lemma \ref{lemma-oneweight}, the irreducible constituents of $I(\sigma,\tau)$ are parametrised by subsets of $\cS(\tau)$, hence $I(\sigma,\tau)$ has \emph{at most} $2^{|\cS(\tau)|}$ irreducible constituents (as we possibly need to forget some fake weights).

\begin{lemma}\label{lemma-auxiliary}
Let $\sigma^c$ be an irreducible constituent of $\mathrm{inj}_{\Gamma}{\sigma}$.  Assume that   $I(\sigma,\sigma^c)$ has $2^{|\cS(\sigma^c)|}$ irreducible constituents. Then, for every irreducible constituent $\tau$ of $I(\sigma,\sigma^c)$, the following statements hold:

(i) the weight $\tau$ is regular;

(ii) there exists a unique weight $\tau^c$ of $I(\sigma,\sigma^c)$ such that $\cS(\tau^c)=\cS(\sigma^c)\backslash \cS(\tau)$;

(iii) the representation $I(\tau,\tau^c)$ exists, and has the same semi-simplification as $I(\sigma,\sigma^c)$.
\end{lemma}

\begin{proof}
Since the results  in the case  $f=1$ are obvious,  we assume $f\geq 2$ in the rest of the proof.

(i) Assume $\tau$ is not regular. Let $\mu\in \cI(x_0,\cdots,x_{f-1})$ be the element corresponding to $\tau$. Then, $\mu_i(r_i)=p-1$ for some $i\in\cS$. Since $0\leq r_i\leq p-2$, this happens only when $r_i=0$ and $\mu_i(x_i)=p-1-x_i$. By definition of $\cI(x_0,\cdots,x_{f-1})$, we have $i+1\in\cS(\mu)$. By Lemma \ref{lemma-oneweight}, there exists a unique element $\mu'\in\cI(x_0,\cdots,x_{f-1})$   which is compatible with $\mu$ and such that $\cS(\mu')=\cS(\mu)\backslash\{i+1\}$. Then $\mu'_{i+1}(x_{i+1})\in\{p-2-x_{i+1},x_{i+1}\}$, hence $\mu'_i(x_i)\in\{x_i,x_i-1,x_i+1\}$; but the condition $i\in\cS(\mu')$ and the compatibility with $\mu$ imply that $\mu'_{i}(x_i)=x_i-1$. Since $r_i=0$,  $\mu'$ does not correspond to a genuine weight, giving a contradiction to the assumption.

(ii) 
The assumption together with Lemma \ref{lemma-oneweight} implies that each subset of $\cS(\sigma^c)$ corresponds to a genuine weight, whence the assertion.

(iii) We may assume $\sigma\neq \sigma^c$, i.e. $\cS(\sigma^c)\neq \emptyset$, therefore $\tau$ and $\tau^c$ are non-isomorphic.   Let $\mu,\mu^c\in\cI(x_0,\cdots,x_{f-1})$ be the elements corresponding to $\tau,\tau^c$. We can write $\mu^c=\nu\circ\mu$ for some (unique) $\nu:=(\nu_i(x_i))_i$ with $\nu_i(x_i)\in\Z\pm x_i$. We claim that  $\nu\in\cI(x_0,\cdots,x_{f-1})$.  First, we check that $\nu_i(x_i)\in\{x_i,x_i\pm1,p-2-x_i,p-2-x_i\pm1\}$. If $i\in \cS(\mu)$, then $i\notin \cS(\mu^c)$, hence
\[\mu_i(x_i)\in\{x_i\pm1, p-2-x_i\pm1\},\ \ \mu^c_i(x_i)\in\{x_i, p-2-x_i\}\]
so that $\mu^c_i(x_i)=\nu_i(\mu_i(x_i))$ with $\nu_i(x_i)\in\{x_i\pm1,p-2-x_i\pm1\}$. Similar statement holds if $i\in\cS(\mu^c)$. If $i\notin \cS(\lambda^c)$, then
\[\mu_i(x_i), \mu^c_i(x_i)\in \{x_i,p-2-x_i\}\]
so that $\mu^c_i(x_i)=\nu_i(\mu_i(x_i))$ with $\nu_i(x_i)\in\{x_i,p-2-x_i\}$.  Next, we verify that the relations (ii), (iii) in the definition of $\cI(x_0,\cdots,x_{f-1})$ are satisfied. This is a direct but tedious check and we only give details for  some special cases. For example, in the case when $i\in\cS(\mu)$ (hence $i\notin\cS(\mu^c)$),  we have $\mu_i^c(x_i)\in\{x_i,p-2-x_i\}$ and the following possibilities:
\begin{enumerate}
\item[--] If $\mu_{i}(x_i)\in\{x_i\pm1\}$ and $\mu_i^c(x_i)=x_i$, then $\nu_i(x_i)\in\{x_{i}\pm1\}$; on the other hand, by the definition of $\cI(x_0,\cdots,x_{f-1})$  we have
\[\mu_{i+1}(x_{i+1}),\ \mu_{i+1}^c(x_{i+1})\in\{x_{i+1},p-2-x_{i+1}\}\]
so that $\nu_{i+1}(x_{i+1})\in\{x_{i+1},p-2-x_{i+1}\}$ which verifies the relation required in the definition of $\cI(x_0,\cdots,x_{f-1})$.
\item[--] If $\mu_i(x_i)\in\{x_i\pm1\}$ and $\mu_i^c(x_i)=p-2-x_i$, then $\nu_i(x_i)\in\{p-2-x_i\pm1\}$; on the other hand,  by the definition of $\cI(x_0,\cdots,x_{f-1})$  we have
\[\mu_{i+1}(x_{i+1})\in\{x_{i+1},p-2-x_{i+1}\},\ \mu_{i+1}^c(x_{i+1})\in\{x_{i+1}\pm1,p-2-x_{i+1}\pm1\}\]
so that $\nu_{i+1}(x_{i+1})\in \{x_{i+1}\pm1,p-2-x_{i+1}\pm1\}$ which verifies the relation required in the definition of $\cI(x_0,\cdots,x_{f-1})$.

\item[--] If $\mu_i(x_i)\in \{p-2-x_i\pm1\}$, we claim that $\mu_i^c(x_i)=x_i$ and so $\nu_i(x_i)\in\{p-2-x_i\pm1\}$. In fact,   otherwise   $\mu_i^c(x_i)=p-2-x_i$, then by   the definition of $\cI(x_0,\cdots,x_{f-1})$
\[\mu_{i+1}(x_{i+1}),\ \mu_{i+1}^c(x_{i+1})\in\{x_{i+1}\pm1,p-2-x_{i+1}\pm1\},\]
so that $i+1\in\cS(\mu)\cap \cS(\mu^c)$ which is impossible. Consequently, $\mu_{i+1}^c(x_{i+1})\in\{x_{i+1},p-2-x_{i+1}\}$ and  so $\nu_{i+1}(x_{i+1})\in\{x_{i+1}\pm1,p-2-x_{i+1}\pm1\}$,
which verifies the relation required in the definition of $\cI(x_0,\cdots,x_{f-1})$.
\end{enumerate}

Having checked that $\mu^c=\nu\circ\mu$ for some $\nu\in\cI(x_0,\cdots,x_{f-1})$, we deduce from \cite[Lem. 3.2]{BP} that $\tau^c$ appears in $\mathrm{inj}_{\Gamma}{\tau}$. Since $\tau$ is regular by (i),  the representation $I(\tau,\tau^c)$ exists; see the discussion after Lemma \ref{lemma-oneweight}.
 Moreover, from the explicit description of $\nu$ we see that $\cS(\nu)=\cS(\lambda^c)$.

More generally, we claim that if $\tau'$ is any irreducible constituent of $I(\sigma,\sigma^c)$, then   $\tau'$ is an irreducible constituent of $\mathrm{inj}_{\Gamma}{\tau}$, i.e. we can write $\mu'=\nu'\circ\mu$ for some $\nu'\in\cI(x_0,\cdots,x_{f-1})$ where $\mu'\in \cI(x_0,\cdots,x_{f-1})$  corresponds to $\tau'$ inside $\mathrm{inj}_{\Gamma}\sigma$. In fact, the above check still works if $i\notin \cS(\mu)\cap \cS(\mu')$, and if $i\in \cS(\mu)\cap\cS(\mu')$, the compatibility with   $\lambda^c$ implies that
\[\mu'_i(x_i)=\nu'_i(\mu_i(x_i)),\ \ \ \mathrm{with}\ \nu'_i(x_i)\in\{x_i,p-2-x_i\}.\]
The relations (ii), (iii) in the definition of $\cI(x_0,\cdots,x_{f-1})$ are verified in the same way as above. This proves the claim. Moreover, we obtain that $\cS(\nu')=(\cS(\mu)\cup\cS(\mu'))\backslash(\cS(\mu)\cap \cS(\mu'))$, which is in particular contained in $\cS(\nu)=\cS(\lambda^c)$.

To finish the proof, we check that $\tau'$  is an irreducible constituent of $I(\tau,\tau^c)$. This amounts to check the compatibility between  $\nu'$ and $\nu$ by \cite[Cor. 4.11]{BP}, which is an easy exercise.
\end{proof}

\subsection{Tame types and their lattices}\label{section-tame}

We call {\em tame types} the irreducible $E$-representations of $\GL_2(\OC_{L})$ that arise by inflation from an irreducible $E$-representation of $\GL_2(\F_q).$ These representations are either principal series, cuspidal, one-dimensional, or twist of the Steinberg representations. In this paper, we only consider the principal series types and the cuspidal types, and all tame types occurring in the following will be assumed to be of this kind. We choose a $\GL_2(\OC_{L})$-invariant lattice $V^\circ$ of a tame type $V$ and consider its reduction $\overline{V^\circ}:=V^{\circ}/\varpi_EV^{\circ}$ whose semi-simplification, denoted by $\overline{V}$, does not depend on the choice of the lattice. We write $\JH(\overline{V})$ for the set of Jordan-H\"older factors of $\overline{V}$. We recall the following well-known results on tame types.

\begin{proposition} \label{prop-tame-1}
Let $V$ be a tame type. Then

(i) $V$ is residually multiplicity free, i.e., each element of  $\JH(\overline{V})$ appears exactly once in $\overline{V}$.

(ii) For any Jordan-H\"older factor $\sigma$ of $\overline{V},$ there is up to homothety a unique $\GL_2(\cO_L)$-stable $\OC_E$-lattice $(V_{\sigma})^\circ$ (resp. $(V^{\sigma})^\circ$) in $V$ such that the cosocle (resp. socle)  of its reduction is $\sigma.$

(iii) Let $\sigma\in\JH(\overline{V})$ and $(V_{\sigma})^{\circ}$ (resp. $(V^{\sigma})^{\circ}$) be as in (ii). If $\tau\in \JH(\overline{V})$ is  such that $\Ext^1_{\Gamma}(\tau,\sigma)\neq0$, then $\tau$ appears in the $1$-st layer of the cosocle filtration (resp. socle filtration) of $\overline{(V_{\sigma})^{\circ}}$ (resp. $\overline{(V^{\sigma})^{\circ}}$).
\end{proposition}

\begin{proof}
(i) is  \cite[Lem. 3.1.1]{EGS}. (ii) is \cite[Lem. 4.1.1]{EGS}.   (iii)  follows from \cite[Thm. 5.1.1]{EGS}, but since our notations here are slightly different, we explain its proof. First note that, by \cite[Cor. 5.6]{BP} (using the assumption $p\geq 3$ when $f=1$), the condition $\Ext^1_{\Gamma}(\tau,\sigma)\neq0$
is equivalent to $\Ext^1_{\Gamma}(\sigma,\tau)\neq0$; if this holds, then they are both 1-dimensional over $\F$ and we have automatically $\sigma\neq \tau$. Moreover, by taking dual and using \cite[Lem. 3.1.1]{EGS}, it suffices to prove the assertion for $(V^{\sigma})^{\circ}$, whose mod $\varpi_E$ reduction  has socle $\sigma$. We can embed $\overline{(V^{\sigma})^{\circ}}$ into $\mathrm{inj}_{\Gamma}\sigma$, hence $\tau$ appear in $\mathrm{inj}_{\Gamma}(\sigma)$ and the representation $I(\sigma,\tau)$ exists; see the discussion after Lemma \ref{lemma-oneweight}. Since $\tau$ appears in $\overline{(V^{\sigma})^{\circ}}$ exactly once, the representation $I(\sigma,\tau)$ is in fact contained in $\overline{(V^{\sigma})^{\circ}}$. We can write down $I(\sigma,\tau)$ explicitly:  by the uniqueness, $I(\sigma,\tau)$ is the unique non-split extension (class) of $\tau$ by $\sigma$, hence of length 2. This implies that $\tau$ appears in  $\mathrm{gr}_1\overline{(V^{\sigma})^{\circ}}$, otherwise $I(\sigma,\tau)$ would have length $\geq 3$.
\end{proof}

The proof of Proposition \ref{prop-tame-1}(iii) has the following consequence.
\begin{corollary}\label{cor-I(sigma,tau)}
Let $V$ be a tame type. For any $\sigma,\tau\in \JH(\overline{V})$, the representation $I(\sigma,\tau)$ exists and is a sub-representation of $\overline{(V^{\sigma})^{\circ}}$.
\end{corollary}

\subsection{Serre weights}\label{section-Serreweights}

In this subsection, we prove a general fact about the set of Serre weights attached to a residual generic representation $\brho$.

From now on, we assume that  $L$ is unramified  over $\Q_p$. Let $f:=[L:\Q_p]$. Let $\brho:\Gal(\bQp/L)\ra \GL_2(\F)$ be a continuous representation. Assume that $\brho$ is \emph{generic} in the sense of \cite[\S11]{BP}, that is, $\brho|_{I(\bQp/L)}$ is isomorphic to one of the following two forms
\begin{enumerate}
\item $\matr{\omega_f^{\sum_{i=0}^{f-1}p^i(r_i+1)}}*01\otimes\eta$ with $0\leq r_i\leq p-3$ for each $i$, and not all $r_i$ equal to $0$ or equal to $p-3$;
\item $\matr{\omega_{2f}^{\sum_{i=0}^{f-1}p^i(r_i+1)}}00{\omega_{2f}^{p^f\sum_{i=0}^{f-1}p^i(r_i+1)}}\otimes\eta$ with $1\leq r_0\leq p-2$, and $0\leq r_i\leq p-3$ for $i>0$.
\end{enumerate}
where $\omega_f$ is the fundamental character of $I(\bQp/L)$ of level $f$ as in \cite[\S11]{BP}. Note that  there are no generic representations if $p=2$, and no reducible generic representations if $p\leq 3$.

To $\brho$ is associated a set of weights, called Serre weights and denoted by $\mathscr{D}(\brho)$ (see \cite[\S11]{BP} or \cite{BDJ}).  {Recall that $\brho$ is {\em tamely ramified} (or \emph{tame} for short)  if and only if it is either reducible split or irreducible.} The genericity of $\brho$ implies that the cardinality of $\mathscr{D}(\brho)$ is $2^f$ if $\brho$ is tame, and is $2^d$ for some $0\leq d\leq f-1$ if $\brho$ is reducible non-split, see \cite[\S11]{BP}. Moreover, if $\brho^{\rm ss}$ denotes the semi-simplification of $\brho$, then we always have $\mathscr{D}(\brho)\subseteq \mathscr{D}(\brho^{\rm ss})$. By \cite[Lem. 2.1.6]{EGS}\footnote{In \cite{EGS}, the genericity condition is the same as in \cite{BP}.}, since $\brho$ is generic, any weight $\sigma\in\mathscr{D}(\brho)$ is regular.

\begin{proposition}\label{prop-tame-2}
There is a tame type $V$ of $\GL_2(\OC_{L})$ such that $\JH(\overline{V})$ identifies with $\mathscr{D}(\overline{\rho}^{\rm ss})$. In particular, $|\JH(\overline{V})|=|\mathscr{D}(\overline{\rho}^{\rm ss})|=2^f.$
\end{proposition}

\begin{proof}
This is \cite[Prop. 4.4]{Br14} or \cite{Dia07}.
\end{proof}

The following result is inspired by \cite[Prop. 4.4]{Br14}  and generalises it.
\begin{proposition}\label{prop-JH=Serre}
  Let $\sigma,\tau$ be two  elements of $\mathscr{D}(\brho)$. Then  the representation $I(\sigma,\tau)$ exists and any irreducible constituent of $I(\sigma,\tau)$ is also an element of $\mathscr{D}(\brho)$.
\end{proposition}

\begin{proof}
If $\brho$ is tame, we take a tame type $V$ such that $\JH(\overline{V})=\mathscr{D}(\brho)$ by Proposition \ref{prop-tame-2}. The  result follows from the proof of Corollary \ref{cor-I(sigma,tau)}.

If $\brho$ is reducible non-split, the existence of $I(\sigma,\tau)$ still follows from Proposition \ref{prop-tame-2} and Corollary \ref{cor-I(sigma,tau)}  by noting that  $\mathscr{D}(\brho)\subset\mathscr{D}(\brho^{\rm ss})$.
The proof of the second assertion is a little subtler.
  First, using the theory of Fontaine-Laffaille module, we attach to $\brho$  a certain subset $J_{\brho}$ of $\cS$ which measures how far $\brho$ is from splitting, see \cite[\S4]{Br14} (we do not need the precise definition here). Remark that the cardinality of  $\mathscr{D}(\brho)$ is exactly $2^{|J_{\brho}|}$ (see \cite[\S4]{Br14}).  Applying \cite[Prop. 4.4(ii)]{Br14} to $J^{\rm min}=\emptyset$ and $J^{\rm max}=\delta(J_{\brho})$ (where $\delta$ is as in \emph{loc. cit.}), we find a (unique) tame character $\chi: I\ra \cO_E^{\times}$ such that the irreducible constituents of $\Ind_{B}^{\Gamma}(\chi^s)$ which are Serre weights of $\brho$ are exactly the weights parametrised by subsets of $\delta(J_{\brho})$ (in the sense of \cite[\S2]{BP}). In other words, if we write $\sigma_{\emptyset}$ for the socle and $\sigma^c_{\emptyset}$ for the weight corresponding to $J^{\rm max}$, then the set $\mathscr{D}(\brho)$ is exactly the set of irreducible constituents of $I({\sigma}_{\emptyset},{\sigma}_{\emptyset}^c)$, because they have the same cardinality. Moreover, the assumption of Lemma \ref{lemma-auxiliary} is satisfied and applying it we deduce that the set $\mathscr{D}(\brho)$ is also the set of irreducible constituents of $I(\sigma,\sigma^c)$ for any weight $\sigma\in\mathscr{D}(\brho)$, where $\sigma^c$ is defined as in Lemma \ref{lemma-auxiliary} relative to $J^{\rm max}$. Since $\tau$ appears in $I(\sigma,\sigma^{c})$, $I(\sigma,\tau)$  is a sub-representation of $I(\sigma,\sigma^{c})$ and the result follows.
\end{proof}

\subsection{The construction of Breuil-Pa\v{s}k\=unas}

Let $\brho$ be as in the previous subsection. Let $D_0(\brho)$ be the finite dimensional $\Gamma$-representation attached to $\brho$ in \cite[\S13]{BP}, that is, $D_0(\brho)$ is the largest  $\F$-representation of $\Gamma$ such that:
\begin{enumerate}
\item[(a)] $\rsoc_{\Gamma}D_0(\brho)=\bigoplus_{\sigma\in\mathscr{D}(\brho)}\sigma$

\item[(b)] any weight of $\mathscr{D}(\brho)$ appears exactly once in $D_0(\brho)$.
\end{enumerate}
We have a $\Gamma$-equivariant decomposition  $D_0(\brho)=\bigoplus_{\sigma\in\mathscr{D}(\brho)}D_{0,\sigma}(\brho)$ by \cite[Prop. 13.1]{BP}, with each $D_{0,\sigma}(\brho)$ satisfying $\rsoc_{\Gamma}D_{0,\sigma}(\brho)=\sigma$. Moreover, $D_0(\brho)$ is multiplicity free by \cite[Cor. 13.5]{BP}. Consequently, $D_{0,\sigma}(\brho)$ is a sub-representation of $A_{\sigma}$  as defined in the beginning of \S\ref{subsection-A} (well-defined since $\sigma$ is regular by \cite[Lem. 2.1.6]{EGS}).

\begin{lemma}\label{lemma-Ext-tau}
Fix a  weight $\sigma\in \mathscr{D}(\brho)$ and let $\tau$ be any weight non-isomorphic to $\sigma$.   If $\Ext^1_{\Gamma}(\tau,D_{0,\sigma}(\brho))\neq0$, then $\tau\in\mathscr{D}(\brho)$. If this is  the case,  then the inclusion $\sigma\hookrightarrow D_{0,\sigma}(\brho)$ induces an isomorphism \[\Ext^1_{\Gamma}(\tau,\sigma)\simto\Ext^1_{\Gamma}(\tau,D_{0,\sigma}(\brho)).\]
\end{lemma}
\begin{proof}
Assume $\Ext^1_{\Gamma}(\tau,D_{0,\sigma}(\brho))\neq0$ and let $M$ be a \emph{non-split} $\Gamma$-extension $0\ra D_{0,\sigma}(\brho)\ra M\ra \tau\ra0$. Then  $\rsoc_{\Gamma}(M)=\rsoc_{\Gamma}(D_{0,\sigma}(\brho))=\sigma$. If $\tau\notin \mathscr{D}(\brho)$, then the representation
\[M\bigoplus \bigl(\bigoplus_{\sigma'\in\mathscr{D}(\brho),\sigma'\neq \sigma}D_{0,\sigma'}(\brho)\bigr)\]
would be strictly larger than $D_0(\brho)$ and verifies the conditions (a), (b) above, contradicting the maximality of $D_0(\brho)$.

We have $\Hom_{\Gamma}(\tau,D_{0,\sigma}(\brho)/\sigma)=0$ as $\tau\in \mathscr{D}(\brho)$, hence the natural morphism $\Ext^1_{\Gamma}(\tau,\sigma)\ra \Ext^1_{\Gamma}(\tau,D_{0,\sigma}(\brho))$ is injective. To show the surjectivity, we need to show that any non-split extension $M$ arises from the pushout of some extension in $\Ext^1_{\Gamma}(\tau,\sigma)$. Since $\rsoc_{\Gamma}M=\rsoc_{\Gamma}D_{0,\sigma}(\brho)$, $M$ can be embedded into $R_{\sigma}$, the injective envelope of $\sigma$ in $\Rep_{\Gamma}$ (noting that the genericity of $\brho$ implies  $\dim_{\F}\sigma\geq2$ for any $\sigma\in\mathscr{D}(\brho)$). Moreover, because $\tau$ is an element of $\mathscr{D}(\brho)$ and distinct from $\sigma$, it does not appear in $D_{0,\sigma}(\brho)$ so that $M$ is  multiplicity free. This shows, using notations in \S\ref{subsection-A}, that $M$ is a sub-representation of $A_{\sigma}$. Consequently, the representation $I(\sigma,\tau)$ is a  sub-representation of $A_{\sigma}$. Since all the irreducible constituents  of $I(\sigma,\tau)$ are Serre weights of $\brho$ by Proposition \ref{prop-JH=Serre}, while by construction only $\sigma$ and $\tau$ are, we obtain that $I(\sigma,\tau)$ is of length 2, i.e., we have $0\ra \sigma\ra I(\sigma,\tau)\ra \tau\ra0$.
\end{proof}

 {From now on, we assume that $\brho$ is tame}.

\begin{lemma}\label{lemma-Ext-sigma}
For each $\sigma\in\mathscr{D}(\brho)$:

(i)  $D_{0,\sigma}(\brho)$ is an alternative sub-representation of $A_{\sigma}$;

(ii) $\Ext^1_{\Gamma}(\sigma,D_{0,\sigma}(\brho))=0$.
\end{lemma}
\begin{proof}
(i) We already remarked that $\sigma$ is regular and $D_{0,\sigma}(\brho)$ is a sub-representation of $A_{\sigma}$. Write
\[\sigma=(\lambda_0(r_0),\cdots,\lambda_{f-1}(r_{f-1}))\otimes{\det}^{e(\lambda)(r_0,\cdots,r_{f-1})}\eta\]
with $\lambda=(\lambda_i(x_i))_i$ as in Lemma 11.2 (if $\brho$ is reducible)  or Lemma 11.4 (if $\brho$ is irreducible) of \cite{BP}. Fix $0\leq i\leq f-1$. If $\mu_i^-(\sigma)$ is defined, the assertion follows
from \cite[Thm. 14.8]{BP}, because  only one of $\mu_i^+(\sigma)$ and $\mu_i^-(\sigma)$ could be compatible with $\mu_{\lambda}$ (notations as in \emph{loc.cit.}). Otherwise, we have either $f=1$ and $\lambda_0(r_0)=p-2$ (hence $\brho$ is irreducible by the genericity of $\brho$), or $f\geq 2$ and $\lambda_{i+1}(r_{i+1})=0$. In the first case, it is direct to check that $D_{0,\sigma}(\brho)$ does not contain $\mu_0^+(\sigma)$, see \cite[\S16]{BP}. We assume $f\geq 2$ and  $\lambda_{i+1}(r_{i+1})=0$ for the rest. We need show that $\mathrm{gr}_1D_{0,\sigma}(\brho)$ does not contain $\mu_{i}^+(\sigma)$, i.e., the element $\mu^+:=(\cdots,p-2-x_{i},x_{i+1}+1,\cdots)$, where $(\mu^+)_{j}(x_j)=x_j$ if $j\notin\{i,i+1\}$,  is \emph{not} compatible with $\mu_{\lambda}$. This also follows from \cite[Thm. 14.8]{BP} by a careful analysis. Indeed,
by the genericity of $\brho$, $\lambda_{i+1}(r_{i+1})=0$ if and only if one of  the following holds:
\begin{enumerate}
\item[(a)] $\brho$ is reducible and, either $r_{i+1}=0$ and $\lambda_{i+1}(x_{i+1})=x_{i+1}$, or $r_{i+1}=p-3$ and $\lambda_{i+1}(x_{i+1})=p-3-x_{i+1}$;
\item[(b)] $\brho$ is irreducible and if $i+1>0$, then either $r_{i+1}=0$ and $\lambda_{i+1}(x_{i+1})=x_{i+1}$, or $r_{i+1}=p-3$ and $\lambda_{i+1}(x_{i+1})=p-3-x_{i+1}$; if $i+1=0$, then either $r_0=1$ and $\lambda_0(x_0)=x_0-1$, or $r_0=p-2$ and $\lambda_0(x_0)=p-2-x_0$.
\end{enumerate}
In all cases, we have $\mu_{\lambda,i+1}(y_{i+1})=p-1-y_{i+1}$, see the paragraph preceding \cite[Thm. 14.8]{BP}. Therefore $\mu^+$ is not compatible with $\mu_{\lambda}$.


For (ii), let $M$ be a \emph{non-split} $\Gamma$-extension
\[0\ra D_{0,\sigma}(\brho)\ra M\ra \sigma\ra0.\]
Then we can view $M$ as a sub-representation of $R_{\sigma}$. By (i), $D_{0,\sigma}(\brho)$ is alternative. Since $\sigma$ is not isomorphic to $\mu_i^+(\sigma)$ or $\mu_i^-(\sigma)$ for any $i$, $M$ is still alternative,
 so that  $M\cap A'_{\sigma}\subseteq A_{\sigma}$ by Lemma \ref{lemma-alternative}. Then Proposition \ref{prop-Q'} implies that $M\subseteq A_{\sigma}$, which gives a contradiction since $A_{\sigma}$ is multiplicity free while $M$ is not.
\end{proof}

\begin{remark}\label{rem-nonsplit}
If $\brho$ is reducible non-split, then $D_{0,\sigma}(\brho)$ is in general not alternative, see \cite[\S16, Example (ia)]{BP}.
\end{remark}

We choose a tame type $V$ such that the Jordan-H\"older factors of $\overline{V}$ are exactly the set $\mathscr{D}(\brho)$; this is possible by Proposition \ref{prop-tame-2}. For $\tau\in \mathscr{D}(\brho)$, let $V^\circ_{\tau}$ be the unique homothety class of lattices in $V$ such that the $\Gamma$-cosocle of $V^\circ_{\tau}/\varpi_E V^\circ_{\tau}$ is $\tau$.

\begin{proposition}\label{prop-quotient-lattice}
Let $\sigma_1,...,\sigma_{m}\in\mathscr{D}(\brho)$ be different weights. Let $\tau\in\mathscr{D}(\brho)$ and assume that $\dim_{\F}\Ext^1_{\Gamma}(\tau,\sigma_i)$=1 for all $1\leq i\leq m$. Let $M$ be the universal extension
\[0\ra \oplus_{i=1}^m\sigma_i\ra M\ra \tau\ra0,\]
i.e. the unique extension of $\tau$ by $\oplus_{i=1}^m\sigma_i$ whose cosocle is $\tau$. Then $M$ is a quotient of $V^{\circ}_{\tau}/\varpi_E V^{\circ}_{\tau}$.
\end{proposition}
\begin{proof}
By definition, $V^{\circ}_{\tau}/\varpi_E V^{\circ}_{\tau}$ has cosocle isomorphic to $\tau$. Proposition \ref{prop-tame-1}(iii) shows that all $\sigma\in \mathscr{D}(\brho)$ such that $\Ext^1_{\Gamma}(\tau,\sigma)\neq0$ must appear in the $1$-st layer of the cosocle filtration of $V^{\circ}_{\tau}/\varpi_E V^{\circ}_{\tau}$, proving the result.
\end{proof}

\begin{corollary}\label{cor-local result}
Let $W$ be a sub-representation of $\oplus_{\sigma\in\mathscr{D}(\brho)}R_{\sigma}$ containing   $D_0(\brho)$. Assume that for each $\tau\in\mathscr{D}(\brho)$, $\Hom_{\Gamma}(V^{\circ}_{\tau}/\varpi_E V^{\circ}_{\tau},W)$ is 1-dimensional over $\F$.  Then  $W=D_0(\brho)$.
\end{corollary}

\begin{proof}
Suppose that $W$ is strictly larger than $D_{0}(\brho)$ and let $\tau$ be a weight appearing in the socle of $W/D_0(\brho)$. Since $W$ and $D_0(\brho)$ have the same socle (both isomorphic to $\oplus_{\sigma\in\mathscr{D}(\brho)}\sigma$), we get a non-split $\Gamma$-extension, denoted by $M$,
\[0\ra D_0(\brho)\ra M\ra\tau\ra0.\]
Lemmas \ref{lemma-Ext-tau} and \ref{lemma-Ext-sigma} show that $\tau\in\mathscr{D}(\brho)$ and this extension is the pushout of a non-split extension $M'$  of $\tau$ by $\oplus_{\sigma\in S}\sigma$, for some subset $S\subset \mathscr{D}(\brho)$ uniquely determined by requiring that the cosocle of $M'$ is $\tau$. By \cite[Cor. 5.6(i)]{BP}, if $\sigma\in \mathscr{D}(\brho)$ is such that $\Ext^1_{\Gamma}(\tau,\sigma)\neq0$, then it is automatically of dimension 1 over $\F$.
Therefore, we may apply Proposition \ref{prop-quotient-lattice} and deduce that the latter extension is a quotient of $V^{\circ}_{\tau}/\varpi_E V^{\circ}_{\tau}$, providing an element of $\Hom_{\Gamma}(V^{\circ}_{\tau}/\varpi_E V^{\circ}_{\tau},W)$ whose image has length at least $2$. However, by assumption  $\Hom_{\Gamma}(V^{\circ}_{\tau}/\varpi_E V^{\circ}_{\tau},W)$ is of dimension 1, and must be spanned by the composite map
\begin{equation}\label{equation-1dim}V^{\circ}_{\tau}/\varpi_E V^{\circ}_{\tau}\twoheadrightarrow \tau\hookrightarrow \rsoc_{\Gamma}W\subseteq W.\end{equation}The contradiction allows to conclude.
\end{proof}

\section{Global input}

We prove Theorem \ref{theorem-main}  in this section. Let $F$ be a totally real extension of $\QM$ in which $p$ is \emph{unramified}, and $\OC_F$ be its ring of integers. For any place $v$ of $F,$ let $F_v$ denote the completion of $F$ at $v$ with ring of integers $\OC_{F_v},$ uniformiser $\varpi_v$ and residue field $k_{v}.$ The cardinality of $k_v$ is denoted by $q_v.$ We let $\AM_{F,f}$ denote the ring of finite ad\`eles of $F.$ If $S$ is a finite set of places of $F,$ we let $\AM^S_{F,f}$ denote the finite ad\`eles outside $S.$ We write $G_F:=\Gal(\Fov/F)$ for the global absolute Galois group of $F,$ and $G_{F_v}:=\Gal(\Fov_v/F_v)$ for the local Galois group at $v.$ We fix an embedding $\Fov\into \Fov_v$ so that $G_{F_v}$ identifies with the decomposition group of $v$ over $F.$ We let $\Frob_v\in G_{F_v}$ denote a (lift of the) geometric Frobenius element, and let $\Art_{F_v}$ denote the local Artin reciprocity map, normalised so that it sends  $\varpi_v$ to $\Frob_v$. The global Artin map is denoted by $\Art_F.$

If $v$ is a place of $F$ over some prime $l.$ An {\em inertial type} for $F_v$ is a two-dimensional $E$-representation $\tau_v$ of the inertia group $I_{F_v}:=I(\overline{F}_v/F_v)$ with open kernel, which may be extended to $G_{F_v}.$ We mainly consider $l=p$ case. Then the tame types considered in \S \ref{section-tame} correspond to the non-scalar tame inertial types under Henniart's inertial local Langlands correspondence \cite{He02}.

Let $D$ be a quaternion algebra with center $F.$ Let $\Sigma$ be the set of finite places where $D$ is ramified. We assume that $\Sigma$ does not contain any place dividing  $p.$ We consider the case that $D$ is ramified at all infinite places ({\em the definite case}), or split at exactly one infinite place ({\em the indefinite case}). \emph{We exclude the case $F=\QM$ and $D=\GL_2.$}

\subsection{Space of modular forms: the definite case}

 Assume $D$ is definite. For $A=\OC_E$ or $\FM,$ a locally constant character $\psi:F^\times\backslash \AM_{F,f}^\times\to A^\times$ and $U=\prod_w U_w\subset(D\otimes_F \AM_{F,f})^\times$ an open compact subgroup such that $\psi|_{U\cap\AM_{F,f}^\times}=1.$ We denote by $S^D_\psi(U,A)$ the space of functions
\[
f:D^\times\backslash (D\otimes_F \AM_{F,f})^\times/U\to A
\]
such that $f(xg)=\psi(x)f(g)$ for any $x\in \AM_{F,f}^\times$ and $g\in (D\otimes_F \AM_{F,f})^\times.$
By taking limit over all the open compact subgroups $U\subset(D\otimes_F \AM_{F,f})^\times$ satisfying $\psi|_{U\cap\AM_{F,f}^\times}=1,$ we obtain an $A$-module
\[
S^D_\psi(A):=\varinjlim_U S^D_\psi(U,A).
\]
which carries a smooth admissible $A$-linear action of $(D\otimes_F \AM_{F,f})^\times$ by $(gf)(h):=f(hg)$ for $g,h\in (D\otimes_F \AM_{F,f})^\times.$

Let $S$ be a finite set of finite places of $F$ which contains $\Sigma,$ the places lying over $p,$ the places at which $U_w$ is not maximal, and the places at which $\psi$ is ramified. We consider the Hecke algebra
\[
\TM^{S}(\OC_E)=\OC_E[\GL_2(\AM^S_{F,f})/ /U^S]=\bigotimes_{w\notin S}{}'\TM_w(\OC_E),
\]
where $U^S=\prod_{w\notin S}\GL_2(\OC_{F_w}),$ and
\[
\TM_w(\OC_E)=\OC_E[\GL_2(F_w)/ /\GL_2(\OC_{F_w})]\cong \OC_E[T_w,S_w^{\pm 1}].
\]
Here, $T_w$ is the Hecke operator corresponding to the double coset
\[
\GL_2(\OC_{F_w}) \left(
                   \begin{array}{cc}
                     \varpi_w & 0 \\
                     0 & 1 \\
                   \end{array}
                 \right)
\GL_2(\OC_{F_w})
\]
and $S_w$ is the one corresponding to
\[
\GL_2(\OC_{F_w})\left(
                   \begin{array}{cc}
                     \varpi_w & 0 \\
                     0 &  \varpi_w\\
                   \end{array}
                 \right)\GL_2(\OC_{F_w}).
\]
The abstract Hecke algebra $\TM^S(\OC_E)$ acts on $S^D_\psi(U,\OC_E)$ via Hecke correspondence, and we let $\TM^S(U,\OC_E)$ denote the image of $\TM^{S}(\OC_E)$ in $\End_{\OC_E}(S^D_\psi(U,\OC_E)).$

Let $\overline{r}:G_{F}\to \GL_2(\FM)$ be a two dimensional continuous totally odd and absolutely irreducible Galois representation. For $w$ a finite place of $F,$ let $\overline{\rho}_w:=\overline{r}|_{G_{F_w}}$ denote the restriction of $\overline{r}$ to $G_{F_w}.$ We choose an $S$ as above which further contains all the ramified places of $\overline{r}.$ We associate to $\overline{r}$ a maximal ideal $\mathfrak{m}^S_{\overline{r}}$ of $\TM^{S}(\OC_E),$ given as the kernel of the map $\TM^S(\OC_E)\to \FM$ sending $T_w$ to $\psi(\Frob_w^{-1})\tr(\overline{\rho}_w(\Frob_w))$ and $S_w$ to $q_w\det(\overline{\rho}_w(\Frob_w))$ for $w\notin S.$

Let
\[
S^D_\psi(U,\FM)[\mathfrak{m}^{S}_{\overline{r}}]:=\{f\in S^D_\psi(U,\FM), ~Tf=0 ~\forall T\in \mathfrak{m}^{S}_{\overline{r}}\}.
\]
By \cite[Lem. 4.6]{BDJ}, the space $S^D_\psi(U,\FM)[\mathfrak{m}^{S}_{\overline{r}}]$ is independent of the choice of $S$, so we denote it $S^D_\psi(U,\FM)[\mathfrak{m}_{\overline{r}}]$. Therefore, we take the limit
\[
S^D_{\psi}(\FM)[\mathfrak{m}_{\overline{r}}]:=\varinjlim_U S^D_\psi(U,\FM)[\mathfrak{m}_{\overline{r}}].
\]

\begin{definition}
We say that $\overline{r}$ is {\em modular} if there is a finite character $\psi:G_F\to \OC_E^\times$ such that $\det\overline{r}=\overline{\varepsilon}^{-1}\overline{\psi}$ with $\varepsilon$ the $p$-adic cyclotomic character, and
\[
S^D_{\psi}(\FM)[\mathfrak{m}_{\overline{r}}]\neq 0.
\]
This is equivalent to demanding the localization $S^D_\psi(U,\FM)_{\mathfrak{m}^S_{\overline{r}}}\neq 0$ for some $U.$
\end{definition}

We recall the following conjecture of \cite{BDJ}.

\begin{conjecture}(\cite[Conj. 4.7]{BDJ})\label{conj-BDJ 1}
Suppose that $\overline{r}$ is modular. There is a $(D\otimes_F \AM_{F,f})^\times$-equivariant isomorphism
\[S^D_{\psi}(\FM)[\mathfrak{m}_{\overline{r}}]\cong \otimes'_w\pi^D_w(\overline{r})\]
where, in the restricted tensor product on the right hand side, $\pi^D_w(\overline{r})$ is the smooth admissible representation of $(D\otimes_F F_w)^\times$ associated to $\overline{\rho}_w$ by Vign\'eras and Emerton when $w\nmid p,$ and when $w|p$,  $\pi^D_w(\overline{r})$ is a smooth admissible representation of $(D\otimes_F F_w)^\times\cong \GL_2(F_w)$
such that if $\sigma$ is a weight, then
$\Hom_{\GL_2(\cO_{F_w})}(\sigma,\pi_w^D(\brho))\neq0 $ if and only if $\sigma\in \mathscr{D}(\overline{\rho}_w)$.
\end{conjecture}

\begin{remark}
The weight part of this conjecture is proved in many cases by Gee \cite{Gee11} and then is completely proved in \cite{GLS} (for $\bar r$ satisfying the usual Taylor-Wiles hypothesis). In \cite{EGS}, the authors proved that if $\brho_{w}$ is generic, then the $\GL_2(\cO_{F_w})$-socle of $\pi_w^{D}(\bar r )$ is exactly $\oplus_{\s\in\mathscr{D}(\overline{\rho}_w)} \s$.
\end{remark}

Let $\s_U$ be a finitely generated continuous representation of some open compact subgroup $U$ of $(D\otimes_F \AM_{F,f})^\times$ over $\OC_E.$ We denote
\[
S^D(\s_U):=\Hom_U(\s_U,S^D_{\psi}(\OC_E)\otimes_{\OC_E} E/\OC_E),
\]
and write $\TM(\s_U)$ for the image of the abstract Hecke algebra in $\End_{\OC_E}(S^D(\s_U)).$

We make the usual Taylor-Wiles assumptions on $\overline{r}:G_F\to \GL_2(\FM).$

\medskip
\noindent{\bf Assumption I.} We assume that $p$ is odd; $\overline{r}$ is modular and $\overline{r}|_{G_{F(\zeta_p)}}$ is absolutely irreducible, where $\zeta_p$ is a primitive $p$-th root of unity; if $p=5,$ assume further that the projective image of $\overline{r}|_{G_{F(\zeta_p)}}$ is not isomorphic to $A_5.$

\medskip

Let $S$ be a subset of finite places of $F$ containing $\Sigma,$ the places above $p$ and the places where $\overline{r}$ or $\psi$ is ramified. We can choose a finite place $w_1\notin S$ with the properties that

$\bullet$ $q_{w_1}\not\equiv 1$ (mod $p$),

$\bullet$ the ratio of the eigenvalues of $\overline{r}(\Frob_{w_1})$ is not equal to $q_{w_1}^{\pm 1}$

$\bullet$ the residue characteristic of $w_1$ is sufficiently large that for any non-trivial root of unity $\zeta$ in a quadratic extension of $F,$ $w_1$ does not divide $\zeta+\zeta^{-1}-2.$

We consider open compact subgroups $U=\prod_w U_w$ of $(D\otimes_F \AM_{F,f})^\times$ for which $U_w$ is maximal if $w\notin S\cup\{w_1\},$ and $U_{w_1}$ is the subgroup of $\GL_2(\OC_{F_{w_1}})$ consisting of elements that are upper-triangular and unipotent modulo $w_1.$ Under these assumptions, $U$ is sufficiently small.  For $\s_U$ as above and each $w$ in $S,$ fix a character $\psi_{\s_U,w}:G_{F_w}\to\OC_E^\times$ such that $\psi_{\s_U,w}(\Frob_w)=\psi(\Frob_w),$ and $\psi_{\s_U,w}|_{I_{F_w}}\circ \Art_{F_w}=\s_U|_{\OC_{F_w}^\times}.$
By Hensel's lemma, there is a unique character $\th_w:G_{F_w}\to \OC_E^\times$ with the properties that $\overline{\th}_w=1$ and $\th_w^2=\psi^{-1}_{\s_U,w}\psi|_{G_{F_w}}.$ We write $\s_U(\th)$ for the twist of $\s_U$ by $\otimes_{w\in S}(\th_w\circ \Art_{F_w}\circ\det),$ where $\det$ is understood as the reduced norm of $D_w^\times$ if $w\in \Sigma.$ We extend $\s_U(\th)$ to an action of $U\cdot \AM_{F,f}^\times$ by letting $\AM_{F,f}^\times$ act via the composite $\AM_{F,f}^\times\to \AM_{F,f}^\times/F^\times\to \cO^\times,$ with the second map the one induced by $\psi\circ \Art_F.$

Let $\s_U(\th)^*$ denote the Pontryagin dual of $\s_U(\th).$ The space $S^D(\s_U)$ can be identified with the space of continuous functions
\[
f:D^\times\backslash (D\otimes_F \AM_{F,f})^\times\to \s_U(\th)^*
\]
such that $f(gu)=u^{-1}f(g)$ for all $g\in(D\otimes_F \AM_{F,f})^\times,u\in U\cdot \AM_{F,f}^\times.$ 
By the weight part of Conjecture \ref{conj-BDJ 1}, for every modular $\overline{r}:G_F\to\GL_2(\FM),$ we can choose $U=\prod_w U_w$ with $U_w=\GL_2(\OC_{F_w})$ for all $w|p,$ and $\s_U:=\otimes_{w|p}\s_w$ with $\s_w$ a Serre weight of $\overline{\rho}_w,$ such that $S^D(\s_U)_{\mathfrak{m}^{S}_{\overline{r}}}\neq 0.$ We define
\[
M_{\overline{r}}(\s_U):=S^D(\s_U)^*_{\mathfrak{m}^{S}_{\overline{r}}},
\]
where $^*$ stands for the Pontryagin dual.

\subsection{Space of modular forms: the indefinite case} Assume $D$ is indefinite. For any open compact subgroup $U\subset (D\otimes_F \AM_{F,f})^\times,$ there is a Shimura curve $X_U$ over $F,$ a smooth projective algebraic curve whose complex points are naturally identified with
\[
D^\times \backslash ((D\otimes_F \AM_{F,f})^\times\times (\CM\backslash \RM))/U.
\]
Let $A$ be $\OC_E$ or $\FM,$ we define
\[
S^D(U,A):=H^1_{\textrm{\'et}}(X_{U,\overline{F}},A),
\]
and
\[
S^D(A):=\varinjlim_U H^1_{\textrm{\'et}}(X_{U,\overline{F}},A)
\]
where the limit is taken over all the open compact subgroups of $(D\otimes_F \AM_{F,f})^\times.$

As in the definite case, we take $S$ a finite set of finite places of $F$ which contains $\Sigma,$ the places lying over $p,$ the places at which $U_w$ is not maximal, and we can form the abstract Hecke algebra $\TM^S(\OC_E)$ acting on $S^D(U,\OC_E)$ and let $\TM^S(U,\OC_E)$ denote its image in $\End_{\OC_E}(S^D(U,\OC_E)).$ Let $\overline{r}:G_{F}\to \GL_2(\FM)$ be as above. For suitable choice of places $S,$ we get a maximal ideal $\mathfrak{m}^S_{\overline{r}}\subset\TM^S(\OC_E),$ and we can define $S^D(U,\FM)[\mathfrak{m}^S_{\overline{r}}],$ $S^D(\FM)[\mathfrak{m}_{\overline{r}}]$ and $S^D(\FM)_{\mathfrak{m}^S_{\overline{r}}}$ as in the definite case, see \cite[Page 150]{BDJ}. If $S^D(\FM)_{\mathfrak{m}^S_{\overline{r}}}\neq 0$ for some $S,$ we say that $\overline{r}$ is modular.

We recall the following conjecture of \cite{BDJ}.

\begin{conjecture}(\cite[Conj. 4.9]{BDJ})\label{conj-BDJ 2}
Suppose that $\overline{r}$ is modular. There is a $G_F\times(D\otimes_F \AM_{F,f})^\times$-equivariant isomorphism
\[
S^D(\FM)[\mathfrak{m}_{\overline{r}}]\cong \overline{r}\otimes\big(\otimes'_w \pi^D_w(\overline{r})\big)
\]
where, in the restricted tensor product on the right hand side, $\pi^D_w(\overline{r})$ is the smooth admissible representation of $(D\otimes_F F_w)^\times$ associated to $\overline{\rho}_w$ by Vign\'eras and Emerton when $w\nmid p,$ and when $w|p$,  $\pi^D_w(\overline{r})$ is a smooth admissible representation of $(D\otimes_F F_w)^\times\cong \GL_2(F_w)$
such that if $\sigma$ is a weight, then
$\Hom_{\GL_2(\cO_{F_w})}(\sigma,\pi_w^D(\brho))\neq0 $ if and only if $\sigma\in \mathscr{D}(\overline{\rho}_w)$.
\end{conjecture}

\begin{remark}
The weight part of this conjecture is also proved in \cite{GLS}(under the usual Taylor-Wiles hypothesis), and the multiplicity one is proved in \cite{EGS}.
\end{remark}

If $\s_U$ is a finitely generated continuous representation of some open compact subgroup $U$ of $(D\otimes_F \AM_{F,f})^\times$ over $\OC_E.$ We denote
\[
S^D(\s_U):=\Hom_U(\s_U,S^D(\OC_E)\otimes_{\OC_E} E/\OC_E),
\]
and write $\TM(\s_U)$ for the image of the abstract Hecke algebra in $\End_{\OC_E}(S^D(\s_U)).$

For the open compact subgroups $U$ under similar conditions as in the definite case, there is a local system $\FC_{\s_U(\th)^*}$ on $X_U$ corresponding to $\s_U(\th)^*$ in the usual way, and $S^D(\s_U)$ can be identified with
$H^1_{\textrm{\'et}}(X_{U,\overline{F}},\FC_{\s_U(\th)^*}).$ 
For modular $\overline{r},$ we can find $S,U,\s_U$ such that $S^D(\s_U)_{\mathfrak{m}^{S}_{\overline{r}}}\neq0.$
We define
\[
M_{\overline{r}}(\s_U):=(\Hom_{\TM(\s_U)_{\mathfrak{m}^{S}_{\overline{r}}}[G_F]}(r_\mathfrak{m},S^D(\s_U)_{\mathfrak{m}^{S}_{\overline{r}}}))^*,
\]
where $r_\mathfrak{m}:G_F\to \GL_2(\TM(\s_U)_{\mathfrak{m}^{S}_{\overline{r}}})$ is the two dimensional continuous Galois representation associated to $\overline{r}$ by Carayol \cite{Car}.

\subsection{Minimal level case}\label{section-minimal}

 We recall the minimal level case as in \cite{BD11} and more generally in \cite{EGS}. Fix a place $v|p.$  Following \cite[\S 6.5]{EGS}, we assume Assumption I and make the following additional assumption:

\medskip

\noindent{\bf Assumption II.} We assume that $p\geq 5,$ $\overline{\rho}_w$ is generic for all places $w|p,$ and if $w\in \Sigma,$ $\overline{\rho}_w$ is not scalar.

\medskip

In fact, the case where $\overline{\rho}_w$ is reducible non-scalar for all $w|p,w\neq v$ is considered in \cite[\S 3.3]{BD11}. This is generalized to the case where $\overline{\rho}_w$ is irreducible in \cite[\S 6.5]{EGS}. The assumption that $\overline{\rho}_w$ is generic excludes the scalar case. We follow the notations in \cite[\S 6.5]{EGS}. Let $\psi$ be the Teichm\"uller lift of $\bar{\varepsilon}\det\bar{r}$, and let $S$ be the union of $\Sigma,$ the places over $p,$ and the places where $\bar{r}$ is ramified.
For each place $w\in S, w\neq v,$ one can define a compact open subgroup $U_w$ of $(\OC_D)_w^\times$ and a finite free $\OC_E$-module $L_w$ with continuous action of $U_w$. We write $L:=\otimes_{w\in S,w\neq v,\OC_E}L_w$ which is a finitely generated continuous representation of $U^{v}:=\prod_{w\in S,w\neq v}U_w\times \prod_{w\notin S}\GL_2(\OC_{F_w}).$ Let $S'$ be the union of the set of places $w|p,w\neq v$ for which $\overline{\rho}_w$ is reducible, the set of places $w\in\Sigma$ for which $\overline{\rho}_w$ is reducible, and the place $w_1.$ For places $w$ in $S',$ one can define the Hecke operator $T_w$ and choose scalars $\beta_w\in \FM^{\times}$ as in {\em loc.cit.}

For any finitely generated $\OC_E$-module $\s_{v}$ with a continuous action of $\GL_2(\OC_{F_{v}}),$ we consider the space of modular forms $S^D(\s_{v}\otimes L)$ with an action of the Hecke algebra $\TM(\s_{v}\otimes L).$ Note that $\s_v\otimes L$ yields a representation of $U$ via the projection $U\to U_S.$ As in {\em loc.cit.} one can extend this action to an action of $\TM(\s_{v}\otimes L)':=\TM(\s_{v}\otimes L)[T_w]_{w\in S'}.$ Denote by $\mathfrak{m}'$ the maximal ideal of $\TM(\s_{v}\otimes L)'$ generated by $\mathfrak{m}^S_{\overline{r}}$ and the elements $T_w-\beta_w$ for $w\in S'.$ We define
\[
S^{D,\min}(\s_{v})_{\mathfrak{m'}}:=S^D(\s_{v}\otimes L)_{\mathfrak{m}'},
\]
and construct $M^{\min}_{\overline{r}}(\s_{v})$ from $S^{D,\min}(\s_{v})_{\mathfrak{m'}}$ in the same way as we constructed $M_{\overline{r}}(\s_U)$ from $S^D(\s_U)_{\mathfrak{m}^{S}_{\overline{r}}}.$

\begin{definition}
We define a smooth admissible representation $\pi^D_{v}(\overline{r})$ of $\GL_2(F_{v})$ as follows: in the definite case,
\[
\pi^D_{v}(\overline{r}):=\Hom_{\FM[U^{v}]}(\otimes_{w\in S,w\neq v}\overline{L}_w,S^D_{\psi}(\FM)[\mathfrak{m}_{\overline{r}}])[\mathfrak{m'}];
\]
and in the indefinite case,
\[
\pi^D_{v}(\overline{r}):=\Hom_{\FM[U^{v}][G_F]}(\overline{r}\otimes\otimes_{w\in S,w\neq v}\overline{L}_w,S^D(\FM)[\mathfrak{m}_{\overline{r}}])[\mathfrak{m'}].
\]
\end{definition}

By definition, for any finite dimensional representation $\s_{v}$ of $\GL_2(\OC_{F_{v}})$ over $\FM$, we have
\begin{equation}\label{equation-isotypic}
\Hom_{\GL_2(\OC_{F_{v}})}(\s_{v},\pi^D_{v}(\overline{r}))\cong M^{\min}_{\bar{r}}(\s_{v})^*[\mathfrak{m}'].
\end{equation}

\begin{remark}\label{remark1}
If the decompositions in Conjectures \ref{conj-BDJ 1} and \ref{conj-BDJ 2} hold, the representation $\pi^D_{v}(\overline{r})$ defined above is the local factor at $v.$ Indeed, Breuil and Diamond \cite[Cor. 3.7.4]{BD11} proved this when $\overline{\rho}_{w}$ is reducible non-scalar for all $w|p$.  By the same arguments, the multiplicity one result of \cite[Thm. 10.1.1]{EGS} (recalled in Thm. \ref{thm-EGS} below) implies the general case when $\overline{\rho}_{w}$ are only assumed to be generic.
\end{remark}

\subsection{Consequences}

In this subsection, we first recall the multiplicity one results of \cite{EGS} for lattices of tame types,  then we combine their results with our local results in \S 2.6 to prove our main theorem.

We continue with the minimal level case in \S \ref{section-minimal}. We let $R^{\Box}_w$ denote the universal framed deformation ring for $\overline{\rho}_w$ over $\OC_E.$ For every $w\in S, w\neq v,$ let $R^{\min}_w$ denote the local deformation ring defined in \cite[Page 51]{EGS}. The ring $R^{\min}_w$ is an $R^{\Box}_w$-algebra and is formally smooth over $\OC_E.$ At the place $v,$ let $\tau_{v}$ be a non-scalar tame inertial type for $I_{F_{v}}$ satisfying the condition that $\det \tau_{v}$ is a lift of $\overline{\varepsilon}\det\overline{\rho}_{v}|_{I_{F_{v}}},$ so that any integral lattice of the tame type $V(\tau_{v})$ admits a central character lifting the character $(\overline{\varepsilon}\det\overline{\rho}_{v}|_{I_{F_{v}}})\circ \Art_{F_{v}}.$ We write $R_{v}^{\Box,\psi|_{G_{F_{v}}}}$ for the quotient of $R^{\Box}_{v}$ corresponding to liftings with determinant $\psi|_{G_{F_{v}}}\varepsilon^{-1}.$ We let $R_{v}^{\Box,\psi|_{G_{F_{v}}},\tau_{v}}$ denote the reduced, $p$-torsion free quotient of $R_{v}^{\Box,\psi|_{G_{F_{v}}}}$ corresponding to potentially crystalline deformations of inertial type $\tau_{v}$ and Hodge-Tate weights $(0,1).$


We let $R^{\loc}$ denote $\widehat{\otimes}_{w\in S}R^{\Box}_w.$ We define $R^{\min,\tau_{v}}:=R_{v}^{\Box,\psi|_{G_{F_{v}}},\tau_{v}}\widehat{\otimes}(\widehat{\otimes}_{w\in S,w\neq v}R^{\min}_w),$ which is formally smooth over $R_{v}^{\Box,\psi|_{G_{F_{v}}},\tau_{v}}.$ Let $R^{\psi}_{\overline{r},S}$ (resp. $R^{\Box,\psi}_{\overline{r},S}$) be the universal (resp. framed) deformation ring for deformations of $\overline{r}$ which are unramified outside $S$ with determinant $\psi\varepsilon^{-1},$ so that $R^{\Box,\psi}_{\overline{r},S}$ is naturally an  $R^{\psi}_{\overline{r},S}$-algebra. We define $R^{\Box,\min}_{\overline{r},S}:=R^{\Box,\psi}_{\overline{r},S}\otimes_{R^{\loc}}R^{\min,\tau_{v}}.$ We also have the corresponding universal deformation ring $R^{\min}_{\overline{r},S}$ defined as the image of $R^{\psi}_{\overline{r},S}$ in $R^{\Box,\min}_{\overline{r},S}.$ The usual Taylor-Wiles-Kisin patching argument \cite[\S 6.4]{EGS}, \cite{GK} applied to the space $M^{\min}_{\overline{r}}(V(\tau_{v})^\circ)$ defined above in both definite and indefinite cases gives us the following data:

$\bullet$ positive integers $g,q,$

$\bullet$ $R_{\infty}^{\min,\tau_{v}},$ a power series ring in $g$ variables over $R^{\min,\tau_{v}},$

$\bullet$ an $\OC_E$-algebra homomorphism $\OC_E[[x_1,\ldots,x_{4\# S+q-1}]]\to R_{\infty}^{\min,\tau_{v}}$ such that
\[
R_{\infty}^{\min,\tau_{v}}/(x_1,\ldots,x_{4\# S+q-1})\cong R^{\min}_{\overline{r},S},
\]

$\bullet$ a coherent sheaf $M^{\min}_{\overline{r},\infty}(V(\tau_{v})^\circ)$ over $R_{\infty}^{\min,\tau_{v}}$ with an isomorphism
\[
M^{\min}_{\overline{r},\infty}(V(\tau_{v})^\circ)/(x_1,\ldots,x_{4\# S+q-1})\cong M^{\min}_{\overline{r}}(V(\tau_{v})^\circ).
\]

For any tame type $V,$ it is observed in \cite[Lem. 3.1.1]{EGS} that $V$ can be defined over an unramified extension of $\QM_p.$ By going through all the above constructions, we can and we do assume that $E$ is an unramified extension of $\QM_p$ in the following. We recall Theorem 10.1.1 of \cite{EGS}.

\begin{theorem}\label{thm-EGS}
Assume Assumptions I and II, and let $\tau_{v}$ be a non-scalar tame inertial type for $I_{F_{v}}.$ For any $\s\in \JH(\overline{V(\tau_{v})}),$ we have the (homothety class of) lattice $(V(\tau_{v})_\s)^\circ$ such that the cosocle of its reduction is $\sigma$ (see Prop. \ref{prop-tame-1} (ii)). Then the coherent sheaf $M_{\overline{r},\infty}^{\min}((V(\tau_{v})_\s)^\circ)$ is free of rank one over $R_{\infty}^{\min,\tau_{v}}.$
\end{theorem}

\begin{corollary}\label{cor-final}
Under the above assumptions, we have
\[
\dim_{\FM}M^{\min}_{\overline{r}}(\overline{(V(\tau_{v})_\s)^\circ})^*[\mathfrak{m'}]\leq 1.
\]
Therefore, we have
\begin{equation}\label{eq-mult one}
\dim_{\FM}\Hom_{\GL_2(\OC_{F_{v}})}(\overline{(V(\tau_{v})_\s)^\circ},\pi^D_{v}(\overline{r}))\leq 1.
\end{equation}
\end{corollary}

\begin{proof} The second inequality follows from the first one by (\ref{equation-isotypic}). Hence, it suffices to prove the first inequality. To simplify the notation, we denote $V(\tau_{v})$ by $V$ in the proof.

Since $M^{\min}_{\overline{r},\infty}$ is a minimal fixed determinant patching functor (see \cite[\S 6]{EGS}), from Theorem \ref{thm-EGS}, we have that $M^{\min}_{\overline{r},\infty}((V_\s)^\circ)$ is a cyclic $R_{\infty}^{\min,\tau_{v}}$-module. Then the $\FM\cong R_{\infty}^{\min,\tau_{v}}/\mathfrak{m}_{R_{\infty}^{\min,\tau_{v}}}$-module
\[
M^{\min}_{\overline{r},\infty}((V_\s)^\circ)/\mathfrak{m}_{R_{\infty}^{\min,\tau_{v}}}M^{\min}_{\overline{r},\infty}((V_\s)^\circ)
\]
is a cyclic $\FM$-module. Therefore,
\[
\dim_{\FM}M^{\min}_{\overline{r},\infty}((V_\s)^\circ)/\mathfrak{m}_{R_{\infty}^{\min,\tau_{v}}}M^{\min}_{\overline{r},\infty}((V_\s)^\circ)\leq 1.
\]
As a dual $\FM$-module of $M^{\min}_{\overline{r},\infty}((V_\s)^\circ)/\mathfrak{m}_{R_{\infty}^{\min,\tau_{v}}}M^{\min}_{\overline{r},\infty}((V_\s)^\circ),$  $M^{\min}_{\overline{r}}(\overline{(V_\s)^\circ})^*[\mathfrak{m'}]$ is also of dimension less or equal to $1.$
\end{proof}

\begin{corollary}
Maintain the same assumptions on $\overline{r}$ as above, and assume further that $\overline{\rho}_{v}$ is tame (i.e. either split or irreducible). Then $\pi^D_{v}(\overline{r})^{K_1}\cong D_0(\overline{\rho}_{v}),$ where $K_1=1+pM_2(\OC_{F_{v}}),$ and $D_0(\overline{\rho}_{v})$ is the representation of $\G:=\GL_2(k_{v})$ constructed by Breuil and Pa\v sk\= unas and recalled in \S 2.6.
\end{corollary}

\begin{proof}
We use freely the notations in \S 2. It suffices to show that $\pi^D_{v}(\overline{r})^{K_1}$ satisfies the hypothesis in Corollary \ref{cor-local result}. From the weight part of Conjectures \ref{conj-BDJ 1} and \ref{conj-BDJ 2}, we have
\[
\soc_{\G}(\pi^D_{v}(\overline{r})^{K_1})=\soc_{\GL_2(\OC_{F_{v}})}\pi^D_{v}(\overline{r})=\oplus_{\s\in\mathscr{D}(\overline{\rho}_{v})}\s.
\]
Hence $\pi^D_{v}(\overline{r})^{K_1}$ embeds into $\oplus_{\s\in\mathscr{D}(\overline{\rho}_{v})}R_\s,$ where   $R_\s$ is the injective envelope of $\s$ in $\Rep_\G$ (since $\dim\sigma\geq 2$), see  the beginning of \S 2. By \cite[Prop. 9.3]{Br14} and Remark \ref{remark1}, we have a $\G$-equivariant injection $D_0(\overline{\rho}_{v})\into\pi^D_{v}(\overline{r})^{K_1}.$ Since by assumption $\brho_{v}$ is tame, we can choose a tame type $V$ with $\JH(\overline{V})=\mathscr{D}(\overline{\rho}_{v})$ by Proposition \ref{prop-tame-2}. For any $\s\in\mathscr{D}(\overline{\rho}_{v}),$ let $V_\s^\circ$ denote the unique $\Gamma$-stable $\cO_E$-lattice of $V$ whose reduction has cosocle $\s.$ Then we have a natural $\Gamma$-equivariant morphism
\[
V_\s^\circ/\varpi_E V_\s^\circ\twoheadrightarrow \s\into\pi^D_{v}(\overline{r})^{K_1},
\]
and hence by (\ref{eq-mult one}), we have
\[
\dim_{\FM}\Hom_{\G}(V_\s^\circ/\varpi_E V_\s^\circ,\pi^D_{v}(\overline{r})^{K_1})= 1.
\]
Finally, Corollary \ref{cor-local result} allows to conclude.
\end{proof}

\textbf{Acknowledgements}
This work was partially supported by China's Recruitement Program of Global Experts and National Center for  Mathematics and Inter disciplinary Sciences, Chinese Academy of Sciences. Part of this work was done during a visit of the second author to University of Rennes 1, Morningside Center of Mathematics and Yau Mathematical Sciences Center at Tsinghua University. He wants to thank these institutions for their hospitality, and he thanks G. Pappas for useful discussions. Finally, we want to thank the referee for his/her very careful reading and various suggestions.

  \[\underline{\hspace{6cm}}\]

\bigskip

\noindent Morningside Center of Mathematics, Academy of Mathematics and Systems Science,
 Chinese Academy of Sciences, University of the Chinese Academy of Sciences
Beijing, 100190,
China.\\
{\it E-mail:} {\ttfamily yhu@amss.ac.cn}\\

\noindent

\noindent  Department of Mathematics, Michigan State University, East Lansing, MI 48824, USA\\
{\it E-mail:} {\ttfamily wanghaoran@math.msu.edu}\\

 \end{document}